\documentclass[reqno,a4paper]{amsart}
\usepackage{mathtools}
\usepackage{enumerate}
\allowdisplaybreaks

\newtheorem{theorem}{Theorem}[section]
\newtheorem{lemma}[theorem]{Lemma}
\newtheorem{corollary}[theorem]{Corollary} 
\newtheorem{proposition}[theorem]{Proposition} 

\theoremstyle{definition}
\newtheorem{definition}[theorem]{Definition}

\theoremstyle{remark}
\newtheorem{remark}[theorem]{Remark}

\numberwithin{equation}{section}

\newcommand*{\ind}[1]{1_{#1}} 
\newcommand*{\indbr}[1]{1_{\{#1\}}} 
\newcommand*{\EE}{\mathbb{E}} 
\newcommand*{\PP}{\mathbb{P}} 
\newcommand*{\RR}{\mathbb{R}} 
\newcommand*{\eps}{\varepsilon}
\DeclareMathOperator{\Ent}{Ent}	
\DeclareMathOperator{\Var}{Var}	
\DeclareMathOperator{\sgn}{sgn} 
\DeclareMathOperator{\cl}{cl} 	
\newcommand*{\law}[1]{\mathcal{L}(#1)} 
		\newcommand*{\mbeta}{\mathcal{M}_{\beta}(m, \sigma ^{\beta +1})}
		\newcommand*{\mzero}{\mathcal{M}_{0}(m, \sigma )}
		\newcommand*{\mone}{\mathcal{M}_{1}(m, \sigma ^2)}
\newcommand*{\mtilde}{\widetilde{m}}
\newcommand*{\vphi}{\varphi}
\newcommand*{\vphitilde}{\widetilde{\varphi}}
		\newcommand*{\Hb}{H_{\beta}}   
		\newcommand*{\Lb}{H^*_{\beta}} 
		\newcommand*{\Hbn}{H^{(n)}_{\beta}}   
		\newcommand*{\Lbn}{H^{(n)*}_{\beta}} 
		\newcommand*{\Hzero}{H_{0}}   
		\newcommand*{\Lzero}{H^*_{0}} 
		\newcommand*{\Hzerodelta}[1]{H_{0,#1}}   
		\newcommand*{\Lzerodelta}[1]{H^*_{0,#1}} 
		\newcommand*{\Hzerodeltan}[1]{H	^{(n)}_{0,#1}}   
		\newcommand*{\Lzerodeltan}[1]{H^{(n)*}_{0,#1}} 
\newcommand*{\step}[1]{\emph{Step~#1.}}
\newcommand*{\negativespace}{\!\!\!} 
\newcommand*{\itemi}{(i)} 
\newcommand*{\itemii}{(ii)}
\newcommand*{\itemiii}{(iii)}
\newcommand*{\itemiv}{(iv)}

\begin{document}

\title[Modified Log--Sobolev Inequalities]{Modified Log--Sobolev Inequalities\\
for Convex Functions on the Real Line.\\
Sufficient conditions}

\author{Rados{\l}aw Adamczak}
\address{Institute of Mathematics, University of Warsaw, Banacha 2, 02--097 Warsaw, Poland.}
\email{R.Adamczak@mimuw.edu.pl}
\thanks{Research partially
supported by the NCN grant no. 2012/05/B/ST1/00412 (R.A.)}

\author{Micha{\l} Strzelecki}
\address{Institute of Mathematics, University of Warsaw, Banacha 2, 02--097 Warsaw, Poland.}
\email{M.Strzelecki@mimuw.edu.pl}

\date{May 20, 2015, revised: December 1, 2015}

\begin{abstract}
We provide a mild sufficient condition for a probability measure on the real line to satisfy a modified log-Sobolev inequality for convex functions, interpolating between the classical log-Sobolev inequality and a Bobkov-Ledoux type inequality. As a consequence we obtain dimension-free two-level concentration results for convex function of independent random variables with sufficiently regular tail decay.

We also provide a link between modified log-Sobolev inequalities for convex functions and weak transport-entropy inequalities, complementing recent work by Gozlan, Roberto, Samson, and Tetali.
\end{abstract}

\maketitle

\section{Introduction and main results}

\subsection{Background} Over the past few decades the concentration of measure phenomenon has become one of the main themes of high dimensional probability and geometric analysis, with applications to limit theorems, non-asymptotic confidence bounds or random constructions of geometric objects with extremal properties. While initial results on concentration of measure concerned mostly deviation bounds for Lipschitz functions of highly regular random variables, the seminal work by Talagrand \cite{TalConcMeasure} has revealed that if one restricts attention to convex Lipschitz functions, dimension-free concentration of measure holds under much weaker conditions, e.g. for all product probability measures with bounded support. Talagrand's approach relied on his celebrated convex distance inequality related to the analysis of isoperimetric problems for product measures. In subsequent papers other authors adapted tools from the classical concentration of measure theory, such as Poincar\'e and log-Sobolev inequalities or transportation and infimum-convolution inequalities \cite{MR1399224,bobkov-goetze,MR1097258,MR1392329,MR1756011,MR2073425,MR2163396}, for proving tail inequalities for convex functions. More recent work \cite{MR3311918,gozlan} presents deep connections between various proposed approaches.

Talagrand's inequality provides a subgaussian tail estimate for convex Lipschitz functions of independent uniformly bounded random variables (with dimension-independent constants). It is natural to ask whether such a result can hold under weaker assumptions on the tails of underlying random variables. One can easily see that it is not enough to assume a uniform subgaussian tail decay and some stronger regularity is necessary \cite{MR1491527,MR2163396}. In \cite{bobkov-goetze} Bobkov and G\"{o}tze show that for a probability measure on the real line dimension-free concentration for convex functions is equivalent to the convex Poincar\'e inequality and further to a certain regularity of the tail (see also~\cite{MR3311918} for  a refinement of this result in the setting of abstract metric spaces). Moreover the concentration rate in this case is at least exponential. In \cite{MR2163396} a~related assumption on the tail is shown to imply the convex log-Sobolev inequality which allows to obtain dimension-free deviation bounds on the upper tail of convex functions.

The goal of this article is to show that tail conditions interpolating between those considered in \cite{bobkov-goetze} and \cite{MR2163396} lead to modified log-Sobolev inequalities for convex functions (convex counterparts of inequalities introduced by Bobkov-Ledoux \cite{MR1440138} and Gentil-Guillin-Miclo \cite{MR2198019,MR2351133}) and thus also two-level deviation bounds interpolating between the subexponential  and subgaussian tail decay, in the spirit of Talagrand's inequalities for the exponential or Weibull distribution
\cite{MR1122615, MR1097258}. Furthermore we demonstrate a connection of modified log-Sobolev inequalities to weak transportation inequalities. In the case of the classical log-Sobolev inequality such a link has been recently discovered by Gozlan, Roberto, Samson, and Tetali \cite{gozlan}.

\subsection{Main results}\label{subsec:main-results}

To formulate our results we need to introduce the following

\begin{definition}\label{def:class-mbeta}
For $\beta\in[0,1]$, $m>0$ and $\sigma\geq 0$ let $\mbeta$ denote the class of probability distributions $\mu$ on $\RR$ for which
\begin{align*}
\forall_{x\ge m} \quad\nu^+([x,\infty)) &\leq \sigma^{\beta+1}\mu([x,\infty)),\\
\forall_{x\le -m} \quad \nu^-((-\infty,x]) &\leq \sigma^{\beta+1}\mu((-\infty,x]),
\end{align*}
where $\nu^+$ is the measure on $[m,\infty)$ with density $x^{\beta}\mu([x,\infty))$ and $\nu^-$ is the measure on $(-\infty,-m]$ with density $|x|^{\beta}\mu((-\infty, x])$.
\end{definition}

As one can easily see, the condition defining the class $\mbeta$ is equivalent to some type of regular tail decay. More precisely, we have the following proposition, the proof of which is deferred to Section~\ref{sec:proofs-of-main-results}.

\begin{proposition}\label{prop:equiv-class-mbeta}
Let $\mu$ be a probability measure on $\RR$. For  $\beta \in [0,1]$ and $m>0$ the following conditions are equivalent.
\begin{enumerate}
\item[\itemi]  There exists $\sigma \geq 0$, such that $\mu\in\mbeta$.
 \item[\itemii] For some $h>0$, $\alpha<1$ and all $x>m$,
\begin{align*}
\mu([x+h/x^{\beta}, \infty)) & \leq \alpha\mu([x,\infty)),\\
\mu((-\infty, -x-h/x^{\beta}]) & \leq \alpha\mu((-\infty,-x]).
\end{align*}
\end{enumerate}
Moreover, the constants in \itemi{} depend only on the constants in \itemii{} and vice versa. For instance, if \itemi{} holds, we can take $h=2\sigma^{\beta+1}$ and $\alpha = 1/2$.
\end{proposition}

Let us recall that a probability measure $\mu$ on $\RR^n$ satisfies the Poincar\'e inequality with constant $C$ for the class $\mathcal{A}$ of real valued functions on $\RR^n$ if for all $f \in \mathcal{A}$,
\begin{displaymath}
\Var f(X) \le C \EE |\nabla f(X)|^2,
\end{displaymath}
where $X$ is a random vector with law $\mu$. The measure $\mu$ satisfies the log-Sobolev inequality with constant $C$ for the class $\mathcal{A}$, if for all $f \in \mathcal{A}$,
\begin{equation}\label{eq:log-Sob-exp}
\Ent e^{f(X)} \le C\EE |\nabla f(X)|^2e^{f(X)},
\end{equation}
where $\Ent Y$ denotes the usual entropy of a non-negative random variable $Y$, i.e.
\begin{displaymath}
\Ent Y = \EE Y\log Y - \EE Y\log \EE Y
\end{displaymath}
if $\EE Y\log Y < \infty$ and $\Ent Y = \infty$ otherwise.

We remark that if $\mathcal{A}$ consists of all smooth functions, one usually states the log-Sobolev inequality in the equivalent form
\begin{equation}\label{eq:log-sob-st}
\Ent f(X)^2 \leq D \EE|\nabla f(X)|^2,
\end{equation}
where $D = 4 C$. We have decided to use the form \eqref{eq:log-Sob-exp}, since our main motivation is the study of concentration properties for convex functions (note that for $\mathcal{A}$ being the class of convex functions \eqref{eq:log-Sob-exp} and \eqref{eq:log-sob-st} are no longer equivalent).

Using Proposition \ref{prop:equiv-class-mbeta} and the results from \cite{bobkov-goetze} one easily obtains that the condition $\mu \in \mzero$ for some $m, \sigma$ is equivalent to the Poincar\'e inequality for convex functions and further to dimension-free subexponential concentration for convex Lipschitz functions on $\RR^n$ equipped with the measure $\mu^{\otimes n}$.
On the other hand as shown in \cite{MR2163396} the condition $\mone$ implies the log-Sobolev inequality for convex functions and a dimension-free subgaussian inequality for upper tails of convex Lipschitz functions on $(\RR^n,\mu^{\otimes n})$, see Corollary \ref{cor:conc-functions2} below.

In what follows  $\|\cdot\|_p$ denotes the $\ell_p^n$ norm, i.e. for $x =(x_1,\ldots,x_n)\in \RR^n$,  $\|x\|_p = (|x_1|^p+\ldots+|x_n|^p)^{1/p}$. For $p = 2$ we will write simply $|x|$ instead of $\|x\|_p$. The symbol $\law{X}$ stands for the law of a random variable $X$.

Our main result is the following theorem.

\begin{theorem}\label{thm:main}
Let $X_1, \ldots, X_n$ be independent random variables such that $\law{X_i}\in\mbeta$ and let $\vphi\colon\RR^n\to\RR$ be a smooth convex function. Then, for $\beta\in(0,1]$,
\begin{align}\label{eq:mod-lsi-beta}
\MoveEqLeft\Ent e^{\vphi(X_1,\ldots, X_n)}\\
&\leq C(\beta, m, \sigma) \EE \Big(|\nabla\vphi(X_1,\ldots, X_n)|^2\lor\|\nabla\vphi(X_1,\ldots, X_n)\|_{\frac{\beta+1}{\beta}}^{\frac{\beta+1}{\beta}}\Big) e^{\vphi(X_1,\ldots, X_n)} .\notag
\end{align}
If $\beta = 0$ and $|\partial_i \vphi(x)| \leq 1/(2(m+3\sigma))$ for $i\in\{1,\ldots, n\}$ and all $x \in \RR^n$, then
\begin{equation}\label{eq:mod-lsi-zero}
\Ent e^{\vphi(X_1,\ldots, X_n)}\leq C(\beta=0, m, \sigma) \EE |\nabla\vphi(X_1,\ldots, X_n)|^2 e^{\vphi(X_1,\ldots, X_n)}.
\end{equation}
\end{theorem}

\begin{remark}
Note that in the case $\beta = 1$, \eqref{eq:mod-lsi-beta} reduces to the classical log-Sobolev inequality
\begin{displaymath}
\Ent e^{\vphi(X_1,\ldots, X_n)} \le C(m, \sigma) \EE |\nabla\vphi(X_1,\ldots, X_n)|^2 e^{\vphi(X_1,\ldots, X_n)}.
\end{displaymath}
The case $\beta \in (0,1)$ corresponds to the modified log-Sobolev inequality introduced and studied by Gentil, Guillin, Miclo \cite{MR2198019} for the class of smooth (not necessarily convex) functions. The inequality \eqref{eq:mod-lsi-zero} for smooth functions was introduced and studied by Bobkov and Ledoux \cite{MR1440138}.
\end{remark}

\begin{remark}
In the above theorem and the following corollaries the assumption about convexity of the function $\vphi$ can be relaxed to separate convexity (cf. \cite{MR1399224}).
\end{remark}

By a version of Herbst's argument we obtain the following corollary, concerning the tail behaviour of smooth convex functions. Below $C$ denotes the constant $C(\beta,m,\sigma)$ from Theorem~\ref{thm:main}.

\begin{corollary} \label{cor:conc-functions}
Let $X_1, \ldots, X_n$ be independent random variables such that $\law{X_i}\in\mbeta$ and let $\vphi\colon\RR^n\to\RR$ be a smooth convex Lipschitz function. Then, for $\beta \in(0,1]$ and all $t\geq 0$,
\begin{multline*}
\PP(\vphi(X_1,\ldots, X_n)\geq  \EE \vphi(X_1,\ldots, X_n) + t)\\
\leq  \exp\Big(- \frac{3}{16} \min\Big\{\frac{t^2}{C\sup |\nabla\vphi|^2}, \frac{t^{1+\beta}}{C^{\beta} \sup \|\nabla\vphi\|_{(\beta+1)/\beta}^{1+\beta}} \Big\}\Big).
\end{multline*}
If $\beta=0$, then for all $t\geq 0$,
\begin{multline*}
\PP(\vphi(X_1,\ldots, X_n)\geq \EE \vphi(X_1,\ldots, X_n) + t)\\
\leq  \exp\Big(- \frac{1}{4} \min\Big\{\frac{t^2}{C\sup |\nabla\vphi|^2}, \frac{t}{(m+3\sigma)\sup \|\nabla\vphi\|_{\infty}} \Big\}\Big).
\end{multline*}
\end{corollary}

Using standard smoothing arguments one can also obtain a result for not necessarily smooth functions, expressed in terms of their Lipschitz constants with respect to the $\ell_2^n$ and $\ell_{1+\beta}^n$ norms.

\begin{corollary} \label{cor:conc-functions2}
Let $X_1, \ldots, X_n$ be independent random variables such that $\law{X_i}\in\mbeta$ and let $\vphi\colon\RR^n\to\RR$ be a convex function such that
\begin{equation*}
|\vphi(x)- \vphi(y)| \leq L_2 \|x-y\|_2
\end{equation*}
and
\begin{equation*}
 |\vphi(x)- \vphi(y)| \leq L_{1+\beta} \|x-y\|_{1+\beta}
\end{equation*}
for all $x,y\in\RR^n$ and some $L_2, L_{1+\beta}<\infty$.	
Then, for $\beta\in(0,1]$ and all $t\geq 0$,
\begin{equation*}
\PP(\vphi(X_1,\ldots, X_n)\geq \EE \vphi(X_1,\ldots, X_n) + t)
\leq  \exp\Big(- \frac{3}{16} \min\Big\{\frac{t^2}{C L_2^2}, \frac{t^{1+\beta}}{C^{\beta} L_{1+\beta}^{1+\beta}} \Big\}\Big).
\end{equation*}
If $\beta=0$, then for all $t\geq 0$,
\begin{equation*}
\PP(\vphi(X_1,\ldots, X_n)\geq \EE \vphi(X_1,\ldots, X_n)+t)
\leq  \exp\Big(- \frac{1}{4} \min\Big\{\frac{t^2}{C L_2^2}, \frac{t}{(m+3\sigma)L_1} \Big\}\Big).
\end{equation*}
\end{corollary}

Finally, we can express concentration in terms of enlargements of convex sets. The proof of this corollary is a modification of the argument for smooth functions, taking into account some additional difficulties related to the fact that in the convex situation one cannot truncate the function (as this operation destroys convexity).

\begin{corollary}\label{cor:conc-sets}
Let $X_1, \ldots, X_n$ be independent random variables such that $\law{X_i}\in\mbeta$ and let $A\in\RR^n$ be a convex set with
$\PP((X_1,\ldots, X_n)\in A)\geq 1/2$. Then, for all $r\geq 0$,
\begin{equation*}
\PP((X_1,\ldots, X_n)\notin A + r^{1/2} B_2 + r^{1/(1+\beta)} B_{1+\beta}) \leq e^{-C' r}
\end{equation*}
for some constant $C'$ depending only on $\beta$ and $C(\beta,m,\sigma)$ from Theorem~\ref{thm:main} (and additionally on $m+3\sigma$ in the case $\beta=0$).
\end{corollary}

The Reader may compare the above corollaries with a general result stated in \cite[Proposition~26, Theorem~27]{barthe-roberto}  about consequences of inequalities similar to \eqref{eq:mod-lsi-beta}
in the classical setting of smooth functions (see also \cite[Example~29]{barthe-roberto}).

\subsection{Relation to other results, and further questions}

Recently, the classical log-Sobolev inequality for convex functions has been shown to be equivalent to certain transport-entropy inequalities with weak transport cost \cite{gozlan}. In Propositions \ref{prop:transport} and \ref{prop:transport-zero} we show that this is also the case for inequalities \eqref{eq:mod-lsi-beta} and \eqref{eq:mod-lsi-zero}. Since the precise formulation of these results requires introducing additional, rather involved notation, we postpone it to Section~\ref{sec:transport}.

Logarithmic Sobolev inequalities for convex functions have one weakness when compared to their counterparts for smooth functions, namely they provide only deviation inequalities for the upper tail of the function (in the classical setting one obtains bounds on the lower tail by applying the inequality to $-\varphi$, an approach which is precluded in our situation, since except for trivial cases $-\varphi$ is not convex).

It turns out that for $\beta = 0$ the condition $\mu \in \mbeta$ is in fact equivalent to two-sided concentration inequality for convex functions. This is a consequence of a recent result by Feldheim, Marsiglietti, Nayar, and Wang \cite{Feldheim} who proved an infimum convolution inequality, which as follows from our Proposition \ref{prop:transport-zero} is stronger than \eqref{eq:mod-lsi-zero}.

\subsubsection{A remark about recent improvements by Gozlan et al.}
During the review process of this manuscript a preprint \cite{gozlan_new} by Gozlan, Roberto, Samson, Shu, and Tetali appeared, in which the Authors extended the results by Feldheim et al. and characterized the infimum-convolution inequality (equivalently weak transportation inequality) for probability measures on the real line and a large class of cost functions. Their condition can be shown to be strictly weaker than ours, while (as follows from Proposition \ref{prop:transport}) the inequalities they consider are stronger than modified log-Sobolev inequalities investigated by us. The condition of \cite{gozlan_new} is expressed in terms of the monotone transport map $U$ between the symmetric exponential distribution and the measure $\mu$ and in the case we consider reads as
\begin{displaymath}
\sup_{x\in \RR} (U(x+u) - U(x)) \le C(1+u)^{1/(\beta+1)}
\end{displaymath}
for some $C < \infty$.
Using the same arguments as in Lemma \ref{lem:ra2} below, one can show that this is implied by $\mu \in \mathcal{M}_\beta(m,\sigma^{\beta+1})$ for some $m,\sigma^{\beta+1}$. On the other hand, for measures supported on the positive half-line, the latter condition is equivalent to
\begin{displaymath}
\sup_{x\in \RR} (U(x+u)^{\beta+1} - U(x)^{\beta+1}) \le D(1+u)
\end{displaymath}
for some $D < \infty$ and so it is not difficult to find a measure which satisfies the condition of \cite{gozlan_new} but does not belong to $\mathcal{M}_\beta(m,\sigma^{\beta+1})$ for any $m,\sigma^{\beta+1}$.

To the best of our knowledge, except for the case of $\beta = 0$, the characterization of measures on the line satisfying the modified log-Sobolev inequality for convex functions is still unknown. For $\beta= 0$ these are exactly the measures $\mu$, which belong to the class $\mzero$ for some $m$ and $\sigma$ (this follows e.g. from Theorem \ref{thm:main} and the characterization of the convex Poincar\'e inequality due to Bobkov and G\"otze \cite{bobkov-goetze}).
We provide a more detailed discussion of the results of \cite{Feldheim,gozlan_new} and their relation to our results in Section~\ref{sec:transport} (see Remark \ref{rem:Feldheim}).

\subsection{Organization of the paper} Section~\ref{sec:proofs-of-main-results} is devoted to the proofs of our main results presented above. In Section~\ref{sec:transport} we discuss connections with transportation inequalities and in Section~\ref{sec:remark-non-convex} we briefly comment on the relation between the class $\mbeta$ and the class of functions satisfying the modified log-Sobolev inequality for all smooth functions.

\subsection{Acknowledgments} We would like to thank Ivan Gentil and Arnaud Guillin who back in 2006 during a trimester Phenomena in High Dimensions held at the IHP in Paris introduced the first-named author to their theory of modified log-Sobolev inequalities and raised a question concerning their convex counterparts. Some ideas used in the proof of the $\beta = 0$ case of our results originated in those conversations. We also thank Piotr Nayar for communicating to us the recent results of \cite{Feldheim}. Finally, we thank the Anonymous Referee for a very careful reading of the first version of this  article, for pointing out that it is possible to obtain hypercontractive estimates in~Section~\ref{sec:transport}, and for many helpful comments.

\section{Proofs of the main results}\label{sec:proofs-of-main-results}

\subsection{Technical lemmas}
In this section we will state and prove several technical lemmas which will be used in the proofs of our main results. Of particular importance is Lemma \ref{lem:ra3}, which constitues the core of the proof of Theorem \ref{thm:main}.

In the proofs we will use the notation $C(\alpha)$ to denote constants depending only on some parameter $\alpha$. The value of such constants may change between occurrences.

Let us start with the proof of Proposition \ref{prop:equiv-class-mbeta} stated in Section \ref{subsec:main-results}.

\begin{proof}[Proof of Proposition \ref{prop:equiv-class-mbeta}]
Assume \itemi{} holds. For $x\geq m$ we have
\begin{align*}
\sigma^{\beta+1}\mu([x,\infty)) & \geq \nu^+([x,\infty))\geq \int_x^{x+2\sigma^{\beta+1}/x^{\beta}} y^{\beta}\mu([y,\infty)) dy \\
& \geq \frac{2\sigma^{\beta+1}}{x^{\beta}}\cdot x^{\beta} \mu([x+2\sigma^{\beta+1}/x^{\beta},
\infty)),
\end{align*}
which clearly implies the first inequality of \itemii{}. The second inequality follows similarly.

Suppose now that \itemii{} is satisfied. For $x\geq m$ define the sequence $a_0=x$,
$a_{n+1}=a_n + h/a_n^{\beta}$. It is easy to see that this sequence is increasing,
$a_n\to\infty$ and	 $a_{n+1}/a_n\to 1$. Therefore
\begin{align*}
\nu^+([x,\infty)) &= \int_x^{\infty} y^{\beta} \mu([y,\infty)) dy \leq \sum_{n=0}^{\infty} (a_{n+1} - a_n) a_{n+1}^{\beta}\mu([a_n,\infty)) \\ &\leq K\sum_{n=0}^{\infty} \alpha^n\mu([a_0,\infty)) = \frac{K}{1-\alpha}\mu([x,\infty)),
\end{align*}
where $K = \sup{} (a_{n+1} - a_n) a_{n+1}^{\beta} = h\sup{} (a_{n+1}/a_n)^{\beta} = h\sup{} (1+h/a_n^{\beta+1})^{\beta} \leq h(1+h/m^{\beta+1})^{\beta}$. We can proceed analogously to obtain
the condition on the left tail.
\end{proof}

\begin{lemma}\label{lem:ra1}
Let $\mu\in\mbeta$. Then for all functions $f:\RR\to[0,\infty)$ which are nonincreasing for $x\leq x_0$ and nondecreasing for $x\geq x_0$ (for some $x_0\in[-\infty,\infty]$) we have
\begin{equation*}
\int_{\mtilde}^{\infty}f(x)x^{\beta}\mu([x,\infty))dx\leq \sigma^{\beta+1}\int_{\RR}
f(x)d\mu(x),
\end{equation*}
where $\mtilde = m\lor ((2\beta)^{1/(\beta+1)}\sigma) +
2\sigma^{\beta+1}/(m\lor((2\beta)^{1/(\beta+1)}\sigma))^{\beta}$.
\end{lemma}

\begin{proof}
First notice that the inequalities of
Definition~\ref{def:class-mbeta} are also satisfied for sets of the form $(x,\infty)$, $x\geq m$. We have
\begin{align}
\int_{\mtilde}^{\infty}f(x)x^{\beta}\mu([x,\infty))dx &=
\int_{\mtilde}^{\infty}\int_0^{\infty} \ind{\{s\leq f(x)\}}x^{\beta}\mu([x,\infty))ds
dx\nonumber\\
&=\int_0^{\infty} \nu^+(\{x\geq\mtilde : f(x)\geq s \}) ds.\label{eq:lem-ra1}
\end{align}
The set $A=\{x\geq\mtilde : f(x)\geq s\}$ is either a (possibly empty) half-line contained in
$[\mtilde,\infty]$ or a disjoint union of such a half-line and an interval $I$ with the left end equal to $\mtilde$. In the former case we have $\nu^+(A)\leq \sigma^{\beta+1} \mu(A) \leq \sigma^{\beta+1} \mu(\{x\in\RR : f(x)\geq s\})$.

In the latter case we denote the right end of $I$ by $t$. Let $a = m\lor
((2\beta)^{1/(\beta+1)}\sigma)$, so that $\mtilde = a +2\sigma^{\beta+1}/a^{\beta}$. Since $t\geq\mtilde\geq a$ we obtain
\begin{multline*}
\nu^+ (A) \leq \nu^+([\mtilde,\infty)) \leq \sigma^{\beta+1} \mu([\mtilde,\infty))
\leq  \sigma^{\beta+1} \mu([a,\mtilde)) \\
\leq \sigma^{\beta+1} \mu([a,t)) \leq \sigma^{\beta+1} \mu(\{ x\leq t : f(x)\geq s \}),
\end{multline*}
where the second inequality follows from the assumption $\mu\in\mbeta$, the third from Proposition~\ref{prop:equiv-class-mbeta} and the definition of $\mtilde$
(which imply that $\mu([a,\mtilde)) = \mu([a,\infty)) - \mu([\mtilde,\infty)) \geq \mu([\mtilde, \infty)) $), and the last one from the observation that $x_0\geq t$ and thus the function $f$ is nonincreasing for $x\leq t$.

Hence, in both cases we have $\nu^+(A) \leq \sigma^{\beta+1} \mu(\{ x\in\RR : f(x)\geq s \})$. Now we can write
\begin{equation*}
\sigma^{\beta+1} \int_0^{\infty} \mu(\{x\in\RR : f(x)\geq s \})ds = \sigma^{\beta+1}\int_{\RR} f(x)d\mu(x),
\end{equation*}
which together with \eqref{eq:lem-ra1} ends the proof.
\end{proof}

\begin{lemma}\label{lem:ra2}
If $X$ is a random variable such that $\law{X}\in\mbeta$ then $\PP(|X|\geq t)\leq C_1(\beta,m,\sigma) e^{-t^{\beta+1} /C_2(\beta,m,\sigma)}$ for all $t\geq 0$. \end{lemma}

\begin{proof}
Obviously it is sufficient to prove the inequality for $t\geq 4m$. Define continuous functions $g(x)=
\int_x^{\infty}y^{\beta}\PP(X\geq y)dy$ and $f(z)=\int_z^{\infty} x^{\beta}g(x)dx$. By the definition of $\mbeta$, for $x\geq m$ we have
\begin{equation*}
g(x)\leq \sigma^{\beta+1}\PP(X\geq x).
\end{equation*}
 Thus, for $z\geq m$,
\begin{equation*}
f(z) = \int_z^{\infty} x^{\beta}g(x)dx\leq \int_z^{\infty} x^{\beta} \sigma^{\beta+1}\PP(X\geq x) dx = \sigma^{\beta+1}g(z).
\end{equation*}
We can rewrite this as
\begin{equation*}
f'(z) \leq -\frac{1}{\sigma^{\beta+1}}z^{\beta}f(z), \quad z\geq m,
\end{equation*}
which gives $f(z)\leq C\exp(-z^{\beta+1}/((\beta+1)\sigma^{\beta+1}))$ with $C$ depending only on $\beta$, $m$ and $\sigma^{\beta+1}$ (note that $f(m)\leq \sigma^{2\beta+2}$). Now, as $g$ is nonincreasing, we have $f(z)\geq \int_z^{2z}y^{\beta}g(y)dy \geq z^{\beta+1}g(2z)$ and hence for $z\geq m$ we get
\begin{equation*}
g(2z)\leq\frac{f(z)}{z^{\beta+1}}  \leq  \frac{C\exp(-z^{\beta+1}/((\beta+1)\sigma^{\beta+1}))}{z^{\beta+1}}.
\end{equation*}
Similarly
\begin{equation*}
\PP(X\geq 4x) \leq \frac{g(2x)}{(2x)^{\beta+1}} \leq \frac{C \exp(-x^{\beta+1}/((\beta+1)\sigma^{\beta+1})) }{2^{\beta+1}x^{2\beta+2}}
\end{equation*}
for $x\geq m$.

We can analogously deal with the lower tail.
\end{proof}

The next two lemmas contain the core of our argument and are counterparts of Lemmas 4 and 3 in \cite{MR2163396}. While Lemma \ref{lem:ra4} follows quite closely the argument from \cite{MR2163396}, the subsequent Lemma \ref{lem:ra3} requires for $\beta < 1$ a significantly new approach.

\begin{lemma}\label{lem:ra4}
Let $\vphi\colon\RR\to\RR$ be a smooth convex Lipschitz function, nonincreasing on $(-\infty,0)$, and $X$ a random variable with $\law{X}\in\mbeta$. Then there exists a constant $C_3(\beta, m, \sigma)$ such that
\begin{align*}
\int_{\{(x,y)\in\RR^2 : xy\leq 0)\}} \vphi'(x)\vphi'(y)e^{\vphi(y)}\PP(X\leq x\land y)&\PP(X\geq x\lor y) dxdy \\
& \leq C_3(\beta, m, \sigma) \EE \vphi'(X)^2 e^{\vphi(X)}.
\end{align*}
\end{lemma}

\begin{proof}
For $\beta > 0 $ we will assume without loss of generality that $m\geq 1$. Let $\mtilde$ be the constant defined in Lemma~\ref{lem:ra1} (note that $\mtilde>1$). For $x<0<y$ we have either $\vphi'(x)\vphi'(y)\leq 0$ or $\vphi'(x)\leq\vphi'(y)<0$ and $\vphi(x)\geq\vphi(y)$, so
\begin{align*}
\MoveEqLeft[5] \int_{-\infty}^0\int_0^{\infty} \vphi'(x)\vphi'(y)e^{\vphi(y)}\PP(X\leq x)\PP(X\geq y) dydx\\
& \leq
\int_{-\infty}^0\int_0^{\infty} \vphi'(x)^2 e^{\vphi(x)}\PP(X\leq x)\PP(X\geq y) dydx\\
& \leq
C(\beta, m, \sigma) \int_{-\infty}^0 \vphi'(x)^2 e^{\vphi(x)}\PP(X\leq x)dx,
\end{align*}
where the last inequality follows from the fact that by Lemma~\ref{lem:ra2}
\begin{equation*}
\int_0^{\infty} \PP(X\geq y) dy = \EE X_{+} \leq C(\beta, m, \sigma).
\end{equation*}
The integral over the set where $x>0>y$ can be treated similarly, namely we have
\begin{align*}
\int_0^{\infty} \int_{-\infty}^0&\vphi'(x)\vphi'(y)e^{\vphi(y)}\PP(X\leq y)\PP(X\geq x) dydx\\
& \leq
\int_0^{\infty}\int_{-\infty}^0 \vphi'(y)^2 e^{\vphi(y)}\PP(X\leq y)\PP(X\geq x) dydx\\
& \leq
C(\beta, m, \sigma) \int_{-\infty}^0 \vphi'(y)^2 e^{\vphi(y)}\PP(X\leq y)dy.
\end{align*}

Now, by Lemma~\ref{lem:ra1} (note that $-X \in \mbeta$ and the function $x\mapsto \vphi'(-x)^2e^{\vphi(-x)}$ is nonincreasing for $x\leq x_0$ and nondecreasing for $x\geq x_0$ for some $x_0\in[-\infty, 0]$),
\begin{align*}
\int_{-\infty}^{-\mtilde} \vphi'(x)^2 e^{\vphi(x)}\PP(X\leq x)dx&\leq \int_{-\infty}^{-\mtilde} \vphi'(x)^2 e^{\vphi(x)}(-x)^{\beta}\PP(X\leq x)dx\\
&\leq \sigma^{\beta+1}\EE \vphi'(X)^2e^{\vphi(X)}.
\end{align*}
Moreover
\begin{equation*}
\int_{-\mtilde}^{0} \vphi'(x)^2 e^{\vphi(x)}\PP(X\leq x)dx= \EE \int_{-\mtilde}^{0} \vphi'(x)^2 e^{\vphi(x)}\ind{\{X\leq x\}}dx\leq \mtilde\EE \vphi'(X)^2e^{\vphi(X)},
\end{equation*}
where the last inequality follows from the fact that $\vphi'(x)^2e^{\vphi(x)}\leq\vphi'(X)^2e^{\vphi(X)}$ for $X\leq x\leq 0$. This ends the proof of the Lemma.
\end{proof}

\begin{lemma} \label{lem:ra3}
Let $\vphi\colon\RR\to\RR$ be a smooth convex Lipschitz function and $X$ a random variable with $\law{X} = \mu \in\mbeta$. If $\beta>0$, then there exists a constant $C_4(\beta, m, \sigma)$ such that
\begin{align}\label{eq:lem-ra3}
\begin{split}
\MoveEqLeft[10]
 \int_0^{\infty}\int_0^{\infty} \vphi'(x)\vphi'(y)e^{\vphi(y)}\PP(X\leq x\land y)\PP(X\geq x\lor y) dxdy \\
& \leq C_4(\beta, m, \sigma) \EE (|\vphi'(X)|^2\lor |\vphi'(X)|^{\frac{\beta+1}{\beta}}) e^{\vphi(X)}.
\end{split}
\end{align}
If $\beta = 0$ and  $\vphi' \leq 1/(2(m+3\sigma))$, then there exists a constant $C_4(\beta=0, m, \sigma)$ such that
\begin{align*}
\MoveEqLeft[12]
\int_0^{\infty}\int_0^{\infty} \vphi'(x)\vphi'(y)e^{\vphi(y)}\PP(X\leq x\land y)\PP(X\geq x\lor y) dxdy \\
 &\leq C_4(\beta=0, m, \sigma) \EE |\vphi'(X)|^2 e^{\vphi(X)}.
\end{align*}
\end{lemma}

\begin{proof}
Let us first notice that the left-hand side of \eqref{eq:lem-ra3} is equal to
\begin{equation*}
\int_0^{\infty}\int_0^y \vphi'(x)\vphi'(y)(e^{\vphi(x)}+e^{\vphi(y)})\PP(X\leq x) \PP(X\geq y) dxdy
\end{equation*}
(since on the set $\{(x,y) : x>y\}$ we can make a change of variables swapping $x$ and $y$).

Since $\vphi$ is convex there exists a point $x_0$ (possibly $0$ or infinity) at which $\vphi$ attains its infimum  on $[0,\infty]$. Let $\mtilde$ be the number defined in Lemma~\ref{lem:ra1}. Obviously, for $\beta>0$ we can assume that $\mtilde\geq 1$ (note that we do not change the value of $\mtilde$ in the case $\beta=0$). We split the outer integral into integrals over intervals $(0,x_0\land \mtilde)$, $(x_0\land \mtilde, x_0\lor\mtilde)$ and $(x_0\lor\mtilde,\infty)$.

\step{1}
Since $\vphi$ is nonincreasing on the interval $(0,x_0\land \mtilde)$ we have
\begin{equation*}
\vphi'(x)\vphi'(y)(e^{\vphi(x)}+e^{\vphi(y)}) \leq 2 \vphi'(x)^2e^{\vphi(x)}
\end{equation*}
for $0\leq x\leq y\leq x_0\land\mtilde$. Therefore
\begin{align}
\MoveEqLeft[4]\int_0^{x_0\land \mtilde}\int_0^y  \vphi'(x)\vphi'(y) (e^{\vphi(x)}+ e^{\vphi(y)}) \PP(X\leq x)\PP(X\geq y) dxdy\nonumber\\
&\leq 2 \int_0^{x_0\land \mtilde}\int_0^y \vphi'(x)^2 e^{\vphi(x)}\PP(X\leq x)\PP(X\geq y) dxdy\nonumber\\
&\leq 2\mtilde\EE \int_0^{x_0\land\mtilde} \vphi'(x)^2 e^{\vphi(x)}\ind{\{X\leq x\}} dx \leq 2\mtilde^2\EE \vphi'(X)^2 e^{\vphi(X)} \label{eq:lem-ra3-int1},
\end{align}
where the last inequality follows from the fact that $\vphi'(x)^2e^{\vphi(x)}\leq \vphi'(X)^2e^{\vphi(X)}$ for $X\leq x\leq x_0$.

\step{2}
To estimate the integral over the interval $(x_0\land \mtilde, x_0\lor\mtilde)$ we consider two cases. If $x_0<\mtilde$, then
\begin{align}
\MoveEqLeft[4]
 \int_{x_0}^{\mtilde}\int_0^y \vphi'(x)\vphi'(y)(e^{\vphi(x)}+e^{\vphi(y)})\PP(X\leq x)\PP(X\geq y) dxdy \nonumber\\
& \leq 2\int_{x_0}^{\mtilde}\int_{x_0}^y \vphi'(y)^2 e^{\vphi(y)} \PP(X\geq y) dx dy\nonumber\\
& \leq 2\mtilde\EE \int_{x_0}^{\mtilde} \vphi'(y)^2 e^{\vphi(y)} \ind{\{X\geq y\}} dy \leq 2\mtilde^2 \EE \vphi'(X)^2e^{\vphi(X)}\label{eq:lem-ra3-int3a}.
\end{align}

If $x_0>\mtilde$, then similarly as above (the third last passage follows by Definition~\ref{def:class-mbeta} and the last by Lemma~\ref{lem:ra1}; for the third passage recall also that $\mtilde\geq 1$ for $\beta>0$)
\begin{align}
\MoveEqLeft \int_{\mtilde}^{x_0}\int_0^y \vphi'(x)\vphi'(y)(e^{\vphi(x)}+e^{\vphi(y)})\PP(X\leq x)\PP(X\geq y) dxdy \nonumber\\
\leq {}& 2\int_{\mtilde}^{x_0} \int_0^y \vphi'(x)^2 e^{\vphi(x)} \PP(X\leq x)\PP(X\geq y) dx dy\nonumber\\
 = {}&
2\int_{0}^{x_0} \int_{x\lor \mtilde}^{x_0}\vphi'(x)^2 e^{\vphi(x)} \PP(X\leq x)\PP(X\geq y) dy dx\nonumber\\
\leq {}&
2\int_{0}^{x_0} \vphi'(x)^2 e^{\vphi(x)} \PP(X\leq x)\int_{x\lor \mtilde}^{\infty} y^{\beta}\PP(X\geq y) dy dx\nonumber\\
\leq {}&
2\sigma^{\beta+1} \int_{0}^{x_0} \vphi'(x)^2 e^{\vphi(x)} \PP(X\leq x)\PP(X\geq x) dx\nonumber\\
\leq {}&
2\sigma^{\beta+1} \Big( \int_{0}^{\mtilde} \vphi'(x)^2 e^{\vphi(x)} \PP(X\leq x)dx + \int_{\mtilde}^{x_0} \vphi'(x)^2 e^{\vphi(x)} x^{\beta}\PP(X\geq x)dx\Big)\nonumber\\
\leq {}&
2\sigma^{\beta+1}( \mtilde + \sigma^{\beta+1}) \EE \vphi'(X)^2e^{\vphi(X)} \label{eq:lem-ra3-int3b}.
\end{align}

\step{3}
It remains to estimate the integral over the interval $(x_0\lor\mtilde, \infty)$. In what follows we can assume that $x_0<\infty$ and $\vphi(x_0)=0$ (since if we subtract $\vphi(x_0)$ from $\vphi$ both sides of \eqref{eq:lem-ra3} will change by a factor of  $\exp(-\vphi(x_0))$). Fix $\delta = 1/(2(\mtilde+\sigma^{\beta+1}))$ and let $x_1 = \sup\{x>x_0 : \vphi'(x) \leq \delta \}$ (with the convention that $x_1 = x_0$ if $\varphi'(x_0) > \delta$). Since $\vphi'(x)\leq 0 \leq\vphi'(y)$ for $0<x\leq x_0 \leq y$ we have
\begin{align}
\MoveEqLeft[6]\int_{x_0\lor \mtilde}^{\infty}\int_0^y  \vphi'(x)\vphi'(y) (e^{\vphi(x)}+ e^{\vphi(y)}) \PP(X\leq x)\PP(X\geq y) dxdy\nonumber\\
\leq {}&
\int_{x_0\lor \mtilde}^{\infty}\int_{x_0}^y  \vphi'(x)\vphi'(y) (e^{\vphi(x)}+ e^{\vphi(y)}) \PP(X\leq x)\PP(X\geq y) dxdy\nonumber\\
 \leq{}&
2\int_{x_0\lor \mtilde}^{\infty}\int_{x_0}^y \vphi'(x)\vphi'(y)  e^{\vphi(y)} \PP(X\geq y) dxdy\nonumber\\
={}&
2\int_{x_0\lor \mtilde}^{\infty}\vphi(y)\vphi'(y)  e^{\vphi(y)} \PP(X\geq y) dy\nonumber\\
\begin{split}
={}&
2\int_{x_0\lor\mtilde}^{x_1\lor \mtilde}\vphi(y)\vphi'(y)  e^{\vphi(y)} \PP(X\geq y) dy\\
&+ 2\int_{x_1\lor \mtilde}^{\infty}\vphi(y)\vphi'(y)  e^{\vphi(y)} \PP(X\geq y) dy.\label{eq:lem-ra3-int2}
\end{split}
\end{align}
We will estimate the two integrals from \eqref{eq:lem-ra3-int2} separately.

\step{3a}
To estimate the first integral we define a convex function
\begin{equation*}
 \vphitilde(x) =
  \begin{cases}
  \vphi(x) 			& \text{if } x\leq x_1, \\
   \vphi(x_1) + (x-x_1)\delta	& \text{if } x > x_1.
  \end{cases}
\end{equation*}
The function $\vphitilde$ is $C^1$-smooth and $0\leq \vphitilde'(x)\leq \delta$ for $x\in (x_0,\infty)$. Moreover  $\vphi(x)=\vphitilde(x)$ and $\vphi'(x)=\vphitilde'(x)$ for $x\in(x_0\lor\mtilde, x_1\lor\mtilde)$ (since this interval is contained in $(x_0, x_1)$ or is an empty set). Also $0\leq \vphitilde(x)\leq\vphi(x)$ and $\vphitilde'(x)^2\leq\vphi'(x)^2$ for $x\in\RR$. Hence (the third inequality follows from Lemma~\ref{lem:ra1})
\begin{align}
\MoveEqLeft[4]\int_{x_0\lor\mtilde}^{x_1\lor\mtilde}\vphi(y)\vphi'(y)  e^{\vphi(y)} \PP(X\geq y) dy \leq \int_{x_0\lor\mtilde}^{\infty}\vphitilde(y)\vphitilde'(y)  e^{\vphitilde(y)} \PP(X\geq y) dy\nonumber\\
&\leq
\frac{1}{2}\int_{x_0\lor\mtilde}^{\infty}(\vphitilde(y)^2 + \vphitilde'(y)^2)  e^{\vphitilde(y)} y^{\beta}\PP(X\geq y) dy\nonumber\\
&\leq
\frac{\sigma^{\beta+1}}{2}\int_{x_0}^{\infty}(\vphitilde(y)^2 + \vphitilde'(y)^2)  e^{\vphitilde(y)}  d\mu(y)\nonumber.\\
&\leq \frac{\sigma^{\beta+1}}{2} \EE \vphitilde(X)^2 e^{\vphitilde(X)}\indbr{X\geq x_0}+ \frac{\sigma^{\beta+1}}{2} \EE \vphi'(X)^2 e^{\vphi(X)}.\label{eq:lem-ra3-int2aa}
\end{align}

In order to complete the estimation of the first integral we have to deal with the expression $\EE \vphitilde(X)^2 e^{\vphitilde(X)}\indbr{X\geq x_0}$. Using integration by parts (recall that $\vphitilde(x_0) = 0$), the Cauchy-Schwarz inequality and the fact that $\vphitilde'$ is bounded by $\delta$ we can write
\begin{align}
\MoveEqLeft\EE \vphitilde(X)^2  e^{\vphitilde(X)}\indbr{X\geq x_0} \nonumber\\
={}&  2\int_{x_0}^{\infty}\vphitilde'(x)\vphitilde(x)e^{\vphitilde(x)}\PP(X\geq x)dx
+ \int_{x_0}^{\infty}\vphitilde(x)^2\vphitilde'(x)e^{\vphitilde(x)}\PP(X\geq x)dx \nonumber\\
\begin{split}
\leq {}&2\Big(\int_{x_0}^{\infty}\vphitilde(x)^2 e^{\vphitilde(x)}\PP(X\geq x)dx\Big)^{1/2}\Big(\int_{x_0}^{\infty}\vphitilde'(x)^2e^{\vphitilde(x)}\PP(X\geq x)dx\Big)^{1/2}\\
&+ \delta\int_{x_0}^{\infty}\vphitilde(x)^2e^{\vphitilde(x)}\PP(X\geq x)dx\label{eq:lem-ra3-int2ab} .
\end{split}
\end{align}

We can now use Lemma~\ref{lem:ra1} to estimate the above integrals by expressions of the type $\EE \vphitilde(X)^2 e^{\vphitilde(X)}\indbr{X\geq x_0}$ and $ \EE \vphi'(X)^2 e^{\vphi(X)}$ respectively. For example
\begin{align*}
\MoveEqLeft[4]\int_{x_0}^{\infty}\vphitilde(x)^2 e^{\vphitilde(x)}\PP(X\geq x)dx \\
\leq {}&\EE\int_{x_0}^{x_0\lor\mtilde}\vphitilde(x)^2 e^{\vphitilde(x)}\ind{\{X\geq x\}}dx +\int_{x_0\lor\mtilde}^{\infty}\vphitilde(x)^2 e^{\vphitilde(x)}x^\beta\PP(X\geq x)dx \\
\leq {}& (\mtilde+\sigma^{\beta+1}) \EE \vphitilde(X)^2 e^{\vphitilde(X)}\indbr{X\geq x_0}
\end{align*}
and similarly (at the end we skip the indicator function and replace $\vphitilde$ by $\vphi$)
\begin{equation*}
\int_{x_0}^{\infty}\vphitilde'(x)^2 e^{\vphitilde(x)}\PP(X\geq x)dx
\leq (\mtilde+\sigma^{\beta+1}) \EE \vphi'(X)^2 e^{\vphi(X)}
\end{equation*}
(note that we use the monotonicity of $\vphitilde$, $\vphitilde'$ to assure that the assumptions of Lemma~\ref{lem:ra1} are satisfied). Plugging this into~\eqref{eq:lem-ra3-int2ab} gives us
\begin{align}
\EE \vphitilde(X)^2  e^{\vphitilde(X)}\indbr{X\geq x_0}
\leq {}&
2(\mtilde+\sigma^{\beta+1})
\Big( \EE \vphitilde(X)^2 e^{\vphitilde(X)}\indbr{X\geq x_0}\Big)^{\frac{1}{2}}
\Big(\EE \vphi'(X)^2 e^{\vphi(X)}\Big)^{\frac{1}{2}}\nonumber\\
& + \delta(\mtilde+\sigma^{\beta+1}) \EE \vphitilde(X)^2 e^{\vphitilde(X)}\indbr{X\geq x_0} \label{eq:lem-ra3-int2ac},
\end{align}
which we solve with respect to $\EE \vphitilde(X)^2  e^{\vphitilde(X)}\indbr{X\geq x_0}$ and (since $\delta (\mtilde+\sigma^{\beta+1}) = 1/2$) arrive at
\begin{equation*}
\EE \vphitilde(X)^2  e^{\vphitilde(X)}\indbr{X\geq x_0}
\leq C(\beta,m,\sigma)\EE \vphi'(X)^2 e^{\vphi(X)}.
\end{equation*}
This allows us to finish the computations  of~\eqref{eq:lem-ra3-int2aa} and finally conclude that
\begin{equation}
\int_{x_0\lor\mtilde}^{x_1\lor\mtilde}\vphi(y)\vphi'(y)  e^{\vphi(y)} \PP(X\geq y) dy \leq C(\beta,m,\sigma) \EE \vphi'(X)^2 e^{\vphi(X)}.\label{eq:lem-ra3-int2ad}
\end{equation}
This ends the estimation of the first integral from~\eqref{eq:lem-ra3-int2}.

\step{3b}
To complete the proof we have to estimate the second integral from~\eqref{eq:lem-ra3-int2}. Notice that if $\beta = 0$, then $x_1=\infty$ since by assumption $\vphi' \leq\delta = 1/(2(\mtilde+\sigma)) = 1/(2(m+3\sigma)) $ and in this case that integral disappears. Therefore we can assume that $\beta > 0$.

Fix $\eps > 0$. By convexity $\vphi(y)^{\beta}\leq\vphi'(y)^{\beta}(y-x_0)^{\beta}\leq \vphi'(y)^{\beta}y^{\beta}$. Using Lemma~\ref{lem:ra1} and Young's inequality with exponents $1/(1-\beta)$ and $1/\beta$ we get
 \begin{align}
\MoveEqLeft[4] \int_{x_1\lor \mtilde}^{\infty}\vphi(y)\vphi'(y)  e^{\vphi(y)} \PP(X\geq y) dy \nonumber\\
 \leq {}&
 \int_{\mtilde}^{\infty}\indbr{y\in (x_1\lor\mtilde,\infty)}\vphi(y)^{1-\beta}\vphi'(y)^{1+\beta}  e^{\vphi(y)} y^{\beta}\PP(X\geq y) dy \nonumber\\
  \leq {}&
\sigma^{\beta+1}  \int_{\RR}\indbr{y\in(x_1\lor\mtilde,\infty)}\vphi(y)^{1-\beta}\vphi'(y)^{1+\beta}  e^{\vphi(y)} d\mu(y) \nonumber\\
 \begin{split}
 \leq{}&
 \eps(1-\beta) \sigma^{\beta+1} \EE \vphi(X) e^{\vphi(X)}\indbr{X\geq x_1\lor\mtilde}\\
  & +  \sigma^{\beta+1} C(\eps,\beta) \EE|\vphi'(X)|^{\frac{1+\beta}{\beta}} e^{\vphi(X)} . \label{eq:lem-ra3-int2ba}
\end{split}
\end{align}

We will estimate $\EE \vphi(X) e^{\vphi(X)}\indbr{X\geq x_1\lor\mtilde}$ by a sum of expressions which appear in the assertion of the lemma and
$ \int_{x_1\lor \mtilde}^{\infty}\vphi(y)\vphi'(y)  e^{\vphi(y)} \PP(X\geq y) dy$. Then we will pick $\eps = \eps(\beta, m,\sigma)$ small enough to get from \eqref{eq:lem-ra3-int2ba} an estimate of the integral $\int_{x_1\lor \mtilde}^{\infty}\vphi(y)\vphi'(y)  e^{\vphi(y)} \PP(X\geq y) dy$ which will end the estimation of the second integral from \eqref{eq:lem-ra3-int2} and the proof of Lemma~\ref{lem:ra3}. Integration by parts gives
\begin{align}\label{eq:lem-ra3-int2bb}
\begin{split}
\MoveEqLeft\EE\vphi(X)e^{\vphi(X)}\indbr{X\geq x_1\lor\mtilde} =  \vphi(x_1\lor\mtilde)e^{\vphi(x_1\lor\mtilde)}\PP(X\geq x_1\lor\mtilde)\\
  & + \int_{x_1\lor\mtilde}^{\infty}\vphi'(y)e^{\vphi(y)}\PP(X\geq y)dy
  + \int_{x_1\lor\mtilde}^{\infty} \vphi(y)\vphi'(y)e^{\vphi(y)}\PP(X\geq y)dy,
\end{split}
\end{align}
so it remains to estimate  $\int_{x_1\lor\mtilde}^{\infty} \vphi'(y)e^{\vphi(y)}\PP(X\geq y)dy$ and the boundary term  $\vphi(x_1\lor\mtilde)e^{\vphi(x_1\lor\mtilde)}\PP(X\geq x_1\lor\mtilde)$. Since $\mtilde\geq 1$ and for $x\geq x_1$ we have $\vphi'(x)\geq \delta$, we can write
\begin{align}
\int_{x_1\lor\mtilde}^{\infty} \vphi'(y)e^{\vphi(y)}\PP(X\geq y)dy \leq {}&
\frac{1}{\delta}\int_{x_1\lor\mtilde}^{\infty} \vphi'(y)^2e^{\vphi(y)}y^{\beta}\PP(X\geq y)dy \nonumber\\
\leq {}&\frac{\sigma^{\beta+1}}{\delta}\EE \vphi'(X)^2 e^{\vphi(X)}\label{eq:lem-ra3-int2bc},
\end{align}

As for the boundary term, we write it as
\begin{align}
\EE \big((\vphi(x_1\lor\mtilde)e^{\vphi(x_1\lor\mtilde)} - e^{\vphi(x_1\lor\mtilde)}+1) +  ( e^{\vphi(x_1\lor\mtilde)}-1)\big)\indbr{X\geq x_1\lor\mtilde}\label{eq:lem-ra3-int2bd}.
\end{align}
It is easy to deal with the second part, since
\begin{align}
\MoveEqLeft[4]\EE ( e^{\vphi(x_1\lor\mtilde)}-1)\indbr{X\geq x_1\lor\mtilde} \leq{} \EE  e^{\vphi(x_1\lor\mtilde)}\indbr{X\geq x_1\lor\mtilde}\nonumber\\
\leq {}& \frac{1}{\delta^2}\EE \vphi'(x_1\lor\mtilde)^2 e^{\vphi(x_1\lor\mtilde)}\indbr{X\geq x_1\lor\mtilde}\leq
\frac{1}{\delta^2} \EE \vphi'(X)^2 e^{\vphi(X)}. \label{eq:lem-ra3-int2be}
\end{align}
To estimate the first part of the boundary term, we use the fact that $(\vphi e^{\vphi}-e^{\vphi}+1)' = \vphi'\vphi e^{\vphi}$ and $\vphi(x_0)e^{\vphi(x_0)}-e^{\vphi(x_0)}+1 = 0$. Therefore
\begin{align}
\MoveEqLeft[2]\EE (\vphi(x_1\lor\mtilde) e^{\vphi(x_1\lor\mtilde)} -e^{\vphi(x_1\lor\mtilde)}+1)\indbr{X\geq x_1\lor\mtilde}\nonumber\\&
=\EE \int_{x_0}^{x_1\lor\mtilde}\vphi'(y)\vphi(y)e^{\vphi(y)}dy \indbr{X\geq x_1\lor\mtilde}\nonumber\\
&\leq
\int_{x_0}^{x_1\lor\mtilde}\vphi'(y)\vphi(y)e^{\vphi(y)} \PP(X\geq y)dy. \label{eq:lem-ra3-int2bf}
\end{align}
Notice that we already estimated $\int_{x_0\lor\mtilde}^{x_1\lor\mtilde}\vphi'(y)\vphi(y)e^{\vphi(y)} \PP(X\geq y)dy$ in Step 3a (see \eqref{eq:lem-ra3-int2ad}), so it is enough to estimate the integral over the interval $(x_0,x_0\lor\mtilde)$. By convexity of $\vphi$ we get
\begin{align}
\MoveEqLeft[6]\int_{x_0}^{x_0\lor\mtilde}\vphi'(y)\vphi(y)e^{\vphi(y)} \PP(X\geq y)dy \leq \int_{x_0}^{x_0\lor\mtilde}\vphi'(y)^2(y-x_0)e^{\vphi(y)} \PP(X\geq y)dy \nonumber\\
&\leq
\mtilde \EE \int_{x_0}^{x_0\lor\mtilde}\vphi'(y)^2 e^{\vphi(y)} \indbr{X\geq y}dy\leq
\mtilde^2 \EE \vphi'(X)^2 e^{\vphi(X)} \label{eq:lem-ra3-int2bg},
\end{align}
where the last inequality follows from the fact that $\vphi'(X)^2 e^{\vphi(X)} \geq \vphi'(y)^2 e^{\vphi(y)}$ for $X\geq y\geq x_0$. This (together with the previous inequalities \eqref{eq:lem-ra3-int2bd}, \eqref{eq:lem-ra3-int2be}, \eqref{eq:lem-ra3-int2bf}) completes the estimation of the boundary term from \eqref{eq:lem-ra3-int2bb}, i.e. we get
\begin{displaymath}
\vphi(x_1\lor\mtilde)e^{\vphi(x_1\lor\mtilde)}\PP(X\geq x_1\lor\mtilde) \le (\frac{1}{\delta^2} + \mtilde^2 + C(\beta,m,\sigma))\EE \vphi'(X)^2 e^{\vphi(X)}.
\end{displaymath}

Therefore we get from   \eqref{eq:lem-ra3-int2bb}, \eqref{eq:lem-ra3-int2bc} the announced estimate
\begin{multline*}
\EE \vphi(X) e^{\vphi(X)}\indbr{X\geq x_1\lor\mtilde} \le C(\beta,m,\sigma)\EE \vphi'(X)^2e^{\vphi(X)} \\
+ \int_{x_1\lor\mtilde}^{\infty} \vphi(y)\vphi'(y)e^{\vphi(y)}\PP(X\geq y)dy.
\end{multline*}
As explained above, plugging this into~\eqref{eq:lem-ra3-int2ba} and taking $\eps$ small enough, we deduce that
\begin{equation}
 \int_{x_1\lor \mtilde}^{\infty}\vphi(y)\vphi'(y)  e^{\vphi(y)} \PP(X\geq y) dy \leq
 C(\beta,m,\sigma) \EE (|\vphi'(X)|^{\frac{1+\beta}{\beta}} +\vphi'(X)^2) e^{\vphi(X)}   \label{eq:lem-ra3-int2bz}
\end{equation}
which ends the estimation of the second integral from \eqref{eq:lem-ra3-int2} (and finishes  Step 3b and also Step 3).

Bringing together the above results of Step 1 (see~\eqref{eq:lem-ra3-int1}), Step 2 (see~\eqref{eq:lem-ra3-int3a}, \eqref{eq:lem-ra3-int3b}) and Step 3 (see~\eqref{eq:lem-ra3-int2}, \eqref{eq:lem-ra3-int2ad} and \eqref{eq:lem-ra3-int2bz}) completes the proof of the lemma. \end{proof}

\subsection{Proofs of Theorem~\ref{thm:main} and Corollaries~\ref{cor:conc-functions}, \ref{cor:conc-functions2}, \ref{cor:conc-sets}} \label{subsec:proofs-of-main-results}

\begin{proof}[Proof of Theorem~\ref{thm:main}]
We will follow Ledoux's approach for bounded random variables. Due to the tensorization property of the entropy (see e.g.~\cite[Chapter 5]{MR1849347}) it is enough to prove the theorem for $n=1$. Also, by a standard approximation argument, we can restrict our attention to convex Lipschitz functions only. Let $Y$ be an independent copy of $X$. By Jensen's inequality we have
\begin{align*}
\Ent e^{\vphi(X)} & =  \EE \vphi(x)e^{\vphi(X)} - \EE e^{\vphi(X)} \log\EE e^{\vphi(X)}\\
& \leq
\frac{1}{2} \EE (\vphi(X) - \vphi(Y))(e^{\vphi(X)} - e^{\vphi(Y)})\\
& =
 \EE (\vphi(X) - \vphi(Y))(e^{\vphi(X)} - e^{\vphi(Y)})\ind{\{X\leq Y\}}\\
& = \EE \int_X^Y \vphi'(x)dx \int_X^Y \vphi'(y)e^{\vphi(y)}dy \ind{\{X\leq Y\}}\\
& = \int_{\RR} \int_{\RR} \vphi'(x) \vphi'(y)e^{\vphi(y)} \PP(X\leq x \land y) \PP(X\geq x \lor y ) dxdy.
\end{align*}
We split the double integral into four integrals depending on the signs of $x$ and $y$ and use Lemmas~\ref{lem:ra3} and~\ref{lem:ra4}  to obtain the desired inequality -- note that  $\law{-X}\in\mbeta$, so we can assume that the infimum of $\vphi$ is attained at some point of $[0,\infty]$ (possibly at $\infty$), which in particular means that the assumptions of  Lemma~\ref{lem:ra4} are satisfied. Moreover, we can use Lemma~\ref{lem:ra3} to handle the integration over $(-\infty,0)^2$ (again by a change of variables and the fact that $\law{-X}\in\mbeta$).
\end{proof}

In the proofs of the corollaries we will use some additional notation and observations.

\begin{remark}\label{rem:poinc}
If we substitute $\eps\vphi$ instead of $\vphi$ into  \eqref{eq:mod-lsi-beta} and divide both sides by $\eps^2$ tending to zero, then we recover the Poincar\'e inequality for convex functions:
\begin{equation*}
\Var \vphi(X_1,\ldots, X_n)\leq 2C(\beta, m, \sigma) \EE |\nabla\vphi(X_1,\ldots, X_n)|^2.
\end{equation*}
A standard approximation argument shows that also for $L$-Lipschitz (but not necessarily smooth) convex functions we have $\Var \vphi(X_1,\ldots, X_n)\leq 2C(\beta, m, \sigma) L^2$.
\end{remark}

\begin{definition}\label{def:H-H*}
For $\beta\in[0,1]$ let $\Hb\colon\RR\to[0,\infty]$ be the function given by $\Hb(t) = \max\{t^2, |t|^{(\beta+1)/\beta}\}$ and let $\Lb\colon\RR\to\RR$ be its Legendre transform.
\end{definition}

For $\beta=0$ the definition of $\Hzero$ should be understood as $\Hzero(t) = t^2\indbr{|t|\leq 1} + \infty\indbr{|t|>1}$. Similarly, all indeterminate expressions below should be interpreted as appropriate limits for $\beta\to 0^+$.
The next lemma sums up some properties of the functions $\Hb$ and $\Lb$.

\begin{lemma}\label{lem:H-H*}
\begin{enumerate}[(a)]
	\item The function $\Lb$ is given by the formula
		\begin{align*}
		\Lb(t)  =
		\begin{cases}
		t^2/4 &\text{if } 0\leq |t|\leq 2,\\
		|t|-1 &\text{if } 2\leq |t| \leq \frac{\beta+1}{\beta},\\
		\frac{1}{\beta} (\frac{\beta}{\beta+1}|t|)^{1+\beta} & \text{if } |t| \geq 				\frac{\beta+1}{\beta}.
		\end{cases}
		\end{align*}
	\item We have $\Lb(t) \geq \frac{3}{16} \min\{t^2, |t|^{1+\beta} \} $.
	\item The derivative of the function $\Lb$ is given by the formula
		\begin{align*}
		\frac{d}{dt }\Lb(t)  =
		\begin{cases}
		|t|\sgn(t)/2 &\text{if } 0\leq |t|\leq 2,\\
		\sgn(t) &\text{if } 2\leq |t| \leq \frac{\beta+1}{\beta},\\
		(\frac{\beta}{\beta+1}|t|)^{\beta}\sgn(t) & \text{if } |t| \geq 				\frac{\beta+1}{\beta}.		
		\end{cases}
		\end{align*}
		\item We have $\frac{d}{dt} \Lb(t) \leq |\frac{d}{dt} \Lb(t)| \leq\min\{|t|, |t|^{\beta} \} $.
\end{enumerate}
\end{lemma}

\begin{proof}[Sketch of the proof.] A straightforward calculation gives (a); (c) and (d) follow directly from it.  To prove (b), first notice that for every $\beta\in[0,1)$ there exists exactly one strictly positive number $t_0$ such that $t_0^2/4 =  \frac{1}{\beta}(\frac{\beta}{\beta+1}t_0)^{1+\beta}$. Since we have
\begin{equation*}
\inf_{\beta\in[0,1]} \frac{t_0-1}{t_0^2/4} = \inf_{t \in [2,4]} \frac{t-1}{t^2/4} = 3/4
\end{equation*}
and the functions $t^2$ and $|t|^{1+\beta}$ are convex, we conclude that \begin{equation*}
\Lb(t) \geq \frac{3}{4} \min\Big\{ t^2/4 ,
\frac{1}{\beta} \Big(\frac{\beta}{\beta+1}|t|\Big)^{1+\beta} \Big\} \geq \frac{3}{16} \min\{t^2 ,
|t|^{1+\beta} \},
\end{equation*}
where we also used the fact that
\begin{equation*}
\inf_{\beta\in[0,1]}\frac{1}{\beta} \Big(\frac{\beta}{\beta+1}\Big)^{1+\beta}  = 1/4.
\qedhere
\end{equation*}
\end{proof}

\begin{proof}[Proof of Corollary~\ref{cor:conc-functions} for $\beta>0$.] We follow the  classical Herbst argument (see e.g. \cite{MR1849347}). Denote $A = \sup\{	|\nabla\vphi(x)| : x\in\RR^n \}$, $B = \sup\{\|\nabla\vphi(x)\|_{(\beta+1)/\beta} : x\in\RR^n \}$ and for $\lambda>0$ define $\Phi(\lambda) = \EE e^{\lambda\vphi(X_1,\ldots, X_n)}$ . Then
\begin{equation*}
\lambda\Phi'(\lambda) = \EE\lambda \vphi(X_1,\ldots, X_n)e^{\lambda\vphi(X_1,\ldots, X_n)}
\end{equation*}
and hence
\begin{align*}
\MoveEqLeft\lambda\Phi'(\lambda) - \Phi(\lambda)\log \Phi(\lambda) = \Ent e^{\lambda\vphi(X_1,\ldots, X_n)}\\
&\leq C \EE \Big((\lambda|\nabla\vphi(X_1,\ldots, X_n)|)^2\lor(\lambda\|\nabla\vphi(X_1,\ldots, X_n)\|_{\frac{\beta+1}{\beta}})^{\frac{\beta+1}{\beta}}\Big) e^{\lambda\vphi(X_1,\ldots, X_n)}\\
&\leq C ((A\lambda)^2\lor(B\lambda)^{(\beta+1)/\beta}) \Phi(\lambda).
\end{align*}
After dividing both sides by $\lambda^2\Phi(\lambda)$ we can rewrite this as
\begin{equation*}
\Big(\frac{1}{\lambda} \log \Phi(\lambda)\Big)' \leq C \big((A\lambda)^2\lor(B\lambda)^{(\beta+1)/\beta}\big)/\lambda^2.
\end{equation*}
Since  the right-hand side is an increasing function of $\lambda$ and $\lim_{\lambda\to 0^+} \frac{1}{\lambda} \log \Phi(\lambda) = \EE \varphi(X_1,\ldots, X_n)$, we deduce from the last inequality that
\begin{equation*}
\frac{1}{\lambda} \log \Phi(\lambda) \leq \EE \varphi(X_1,\ldots, X_n) +  C ((A\lambda)^2\lor(B\lambda)^{(\beta+1)/\beta})/\lambda,
\end{equation*}
which is equivalent to
\begin{equation*}
\EE e^{\lambda\vphi(X_1,\ldots, X_n)} \leq \exp(\lambda\EE \varphi(X_1,\ldots, X_n) +  C ((A\lambda)^2\lor(B\lambda)^{(\beta+1)/\beta})).
\end{equation*}

Therefore from Chebyshev's inequality we get
\begin{align*}
\PP(\vphi(X_1,\ldots, X_n)\geq t+ \EE \vphi(X_1,\ldots, X_n))&\leq \frac{\EE e^{\lambda\vphi(X_1,\ldots, X_n)}}{\exp(\lambda\EE \varphi(X_1,\ldots, X_n) + \lambda t)}\\
&\leq  \exp( - \lambda t +  C ((A\lambda)^2\lor(B\lambda)^{(\beta+1)/\beta}))).
\end{align*}
Now we can optimize the right-hand side with respect to $\lambda$.  Let $K$ and $L$ be such that $A = K^{1/2}L$, $B = K^{\beta/(\beta+1)}L$. We have
\begin{equation*}
((A\lambda)^2\lor(B\lambda)^{(\beta+1)/\beta})) = K((L\lambda)^2\lor(L\lambda)^{(\beta+1)/\beta})) = K\Hb(L\lambda)
\end{equation*}
and hence
\begin{equation*}
\PP(\vphi(X_1,\ldots, X_n)\geq t + \EE \vphi(X_1,\ldots, X_n))\leq  \exp( - CK \Lb(t/CKL)).
\end{equation*}
 Using Lemma~\ref{lem:H-H*} and the definitions of $K$ and $L$  we get
\begin{align*}
\PP(\vphi(X_1,\ldots, X_n)\geq t + \EE \vphi(X_1,\ldots, X_n))
\leq  \exp\Big(- \frac{3}{16} \min\Big\{\frac{t^2}{CA^2}, \frac{t^{1+\beta}}{ C^{\beta} B^{1+\beta}} \Big\}\Big),
\end{align*}
which is the assertion of the lemma.
\end{proof}

\begin{proof}[Proof of Corollary~\ref{cor:conc-functions} for $\beta=0$.] We proceed exactly as in the case $\beta>0$, with the only difference that we have the additional constraint $\lambda \leq 1/(2B(m+3\sigma))$ since the gradient of the function $\lambda\vphi$ has to be small. We arrive at
\begin{equation*}
\PP(\vphi(X_1,\ldots, X_n)\geq t+ \EE \vphi(X_1,\ldots, X_n))\leq  \exp( - \lambda t +  C A^2\lambda^2).
\end{equation*}
Now we can optimize the right-hand side with respect to $\lambda\in[0,  1/(2B(m+3\sigma))]$. For $t/(2CA^2) > 1/(2B(m+3\sigma))$ it is best to take $\lambda = 1/(2B(m+3\sigma))$ for which the polynomial becomes
\begin{multline*}
- \lambda t +  C A^2\lambda^2 = -t/(2B(m+3\sigma)) + C A^2/ (2B(m+3\sigma))^2 \\
\leq - t/(4B(m+3\sigma)) = - \min\{ t/(4B(m+3\sigma)),t^2/(4CA^2) \}.
\end{multline*}
For $t/(2CA^2) \leq 1/(2B(m+3\sigma))$ we put $\lambda = t/(2CA^2)$, which gives
\begin{multline*}
- \lambda t +  C A^2\lambda^2 = -t^2/(2CA^2) + t^2/(4CA^2) \\
\leq - t^2/(4CA^2) = - \min\{ t/(4B(m+3\sigma)),t^2/(4CA^2) \}.
\end{multline*}
The assertion of the lemma follows from those two inequalities .
\end{proof}

\begin{proof}[Proof of Corollary~\ref{cor:conc-functions2}] We will consider only the case $\beta>0$, the case $\beta = 0$ is similar. The function $\vphi$ is differentiable almost everywhere (with respect to the Lebesgue measure) and
\begin{equation*}
\langle \nabla \vphi(x), h \rangle \leq \vphi(x+h) - \vphi(x) \quad \text{a.e.}
\end{equation*}
After plugging $h=\nabla \vphi(x)$ and $h=(|\partial_i \vphi(x)|^{1/\beta}\sgn( \partial_i\vphi(x)))_{i=1}^{n}$, and using the Lipschitz conditions we arrive at
\begin{equation*}
\|\nabla\vphi(x)\|_2^2 \leq L_2^2, \quad
\|\nabla\vphi(x)\|_{(\beta+1)/\beta}^{1+\beta} \leq L_{1+\beta}^{1+\beta} \quad \text{a.e.}
\end{equation*}

Let $\vphi_{\eps}$ be the convolution of $\vphi$ and a Gaussian kernel, i.e. $\vphi_{\eps}(x) = \EE\vphi(x+\sqrt{\eps}G)$, where $G\sim\mathcal{N}(0,I)$. The function $\vphi_{\eps}$ is smooth and convex, and inherits from $\vphi$ the estimates of the gradient (which are now satisfied for all $x\in\RR^n$). Therefore from Corollary~\ref{cor:conc-functions} we get
\begin{equation*}
\PP(\vphi_{\eps}(X_1,\ldots, X_n)\geq t + \EE \vphi_{\eps}(X_1,\ldots, X_n))
\leq  \exp\Big(- \frac{3}{16} \min\Big\{\frac{t^2}{C L_2^2}, \frac{t^{1+\beta}}{C^{\beta} L_{1+\beta}^{1+\beta}} \Big\}\Big).
\end{equation*}
Moreover $|\vphi_{\eps}(x)-\vphi(x)| \leq L_2\sqrt{\eps} \EE \|G\|_2$ and hence $\vphi_{\eps}$ converges uniformly to $\vphi$ as $\eps$ tends to zero. This observation ends the proof of the lemma.	
\end{proof}

We will now pass to the proof of Corollary~\ref{cor:conc-sets}.

\begin{proof}[Proof of Corollary~\ref{cor:conc-sets}.] We will present a detailed proof only for the case $\beta>0$, the case $\beta = 0$ is similar. Throughout the proof we denote by $C(\beta,m,\sigma)$ the constant from Theorem~\ref{thm:main} and by $x_i$, $i=1,\ldots, n$, the coordinates of the vector $x\in\RR^n$.
The function $\Phi(x) = \inf_{a\in A} \sum_{i=1}^n \Lb(x_i-a_i)$ is convex. Therefore,  the set
\begin{equation*}
\widetilde{A} = A + \{ x\in\RR^n : \sum_{i=1}^n \Lb(x_i)<3r/16\} = \{x\in\RR^n : \Phi(x)<3r/16 \}
\end{equation*}
is open and convex.

By Rademacher's theorem, the function $\Phi$ is differentiable almost everywhere (with respect to the Lebesgue measure) and hence we can choose a dense set $D\subset\widetilde{A}$ consisting of points for which $\nabla\Phi$ exists. We define
\begin{equation*}
\vphi(x) = \sup_{y\in D} \{\Phi(y) + \langle\nabla\Phi(y), x-y\rangle \}.
\end{equation*}
This function is convex as a supremum of linear functions. Moreover $\vphi\leq\Phi$ and $\vphi=\Phi$ on the set $\cl\widetilde{A} = \{x\in\RR^n : \Phi(x)\leq3r/16 \}$.

Note that the set $\{x\in\RR^n : \vphi(x)<3r/16 \}$ is open and convex, contains the set $\widetilde{A}$ (since on $\widetilde{A}$ we have $\vphi = \Phi < 3r/16$), and is disjoint with the set $\cl\widetilde{A}\setminus\widetilde{A} = \{x\in\RR^n : \Phi(x)=3r/16 \}$ (since on $\cl\widetilde{A}\setminus\widetilde{A}$ we have $\vphi = \Phi = 3r/16$). Therefore $\{x\in\RR^n : \vphi(x)<3r/16 \}= \widetilde{A}$ and hence
\begin{multline}\label{eq:incl-balls}
\{x\in\RR^n : \vphi(x)<3r/16 \}= A + \{ x\in\RR^n : \sum_{i=1}^n \Lb(x_i)<3r/16\}\\
\subset A + \{ x\in\RR^n : \sum_{i=1}^n x_i^2\land |x_i|^{1+\beta} <r\}\subset A + r^{1/2}B_2 + r^{1/(1+\beta)}B_{1+\beta},
\end{multline}
where we used Lemma~\ref{lem:H-H*} and the observation that if $\sum_{i=1}^n x_i^2\land |x_i|^{1+\beta} <r$, then $x$ can be written as $y+z$, $y\in r^{1/2}B_2$, $z\in r^{1/(1+\beta)}B_{1+\beta}$ (since we can put $y_i=x_i$ if $x_i^2<|x_i|^{1+\beta}$,  $y_i=0$ otherwise, and $z_i=x_i$ if $x_i^2\geq|x_i|^{1+\beta}$, $z_i=0$ otherwise).

We will show that the inclusion
\begin{equation*}
\{x\in\RR^n : \vphi(x)<r/16+\EE\vphi(X_1,\ldots,X_n) \} \subset \{x\in\RR^n : \vphi(x)<3r/16 \}
\end{equation*}
holds for $r$ big enough (greater than some constant depending only on $C(\beta,m,\sigma)$) and that $\vphi$ satisfies
\begin{equation*}
|\vphi(x)- \vphi(y)| \leq L_2 \|x-y\|_2, \quad |\vphi(x)- \vphi(y)| \leq L_{1+\beta} \|x-y\|_{1+\beta}
\end{equation*}
with
\begin{equation*}
L_2^2 = \sup_{y\in D } \|\nabla\Phi(y)\|_2^2 \leq  16r/3, \quad L_{1+\beta}^{1+\beta} = \sup_{y\in D } \|\nabla\Phi(y)\|_{(\beta+1)/\beta}^{1+\beta} \leq 16r^{\beta}/3.
\end{equation*}
By the preceding Corollary~\ref{cor:conc-functions2} the above claims together with~\eqref{eq:incl-balls} will give
\begin{equation*}
\PP((X_1,\ldots, X_n)\notin  A + r^{1/2}B_2 + r^{1/(1+\beta)}B_{1+\beta} )
\leq  \exp(- C' r)
\end{equation*}
for $r$ big enough and some constant $C'$ depending only on $C(\beta,m,\sigma)$ and $\beta$. By decreasing $C'$ we will be able to guarantee that this inequality holds for all $r>0$ (recall that $\PP((X_1,\ldots, X_n)\in  A)\geq 1/2$).

We estimate the Lipschitz constants first. For $y\in D\subset\widetilde{A}$ we can choose $a^y\in\cl A$ such that
\begin{equation*}
\Phi(y) = \inf_{a\in A} \sum_{i=1}^n \Lb(y_i-a_i) = \sum_{i=1}^n \Lb(y_i-a^y_i).
\end{equation*}
Note that then $y-a^y\in\{ x\in\RR^n : \sum_{i=1}^n \Lb(x_i)\leq 3r/16\}$.
Let $e^i$ be the $i$-th vector of the canonical basis in $\RR^n$. We have
\begin{multline*}
\Phi(y+te^i) - \Phi(y) = \inf_{a\in A} \sum_{j=1}^n \Lb(y_j+t\indbr{i=j}-a_j) -  \sum_{j=1}^n \Lb(y_j-a^y_j) \\
\leq \Lb(y_i+t-a^y_i) - \Lb(y_i-a^y_i).
\end{multline*}
The function $t\mapsto \Lb(y_i+t-a^y_i)-\Lb(y_i-a^y_i) -\Phi(y+te^i)+\Phi(y)$ attains its minimum for $t=0$ and hence $\partial_i \Phi(y) = (\Lb)'(y_i - a^y_i) $. Using Lemma~\ref{lem:H-H*} we arrive at $|\partial_i \Phi(y)| \leq |y_i - a^y_i|\land|y_i - a^y_i|^{\beta}$.
Therefore,
\begin{multline*}
\|\nabla\Phi(y)\|_2^2 \leq \sum_{i=1}^n (|y_i - a^y_i|\land |y_i - a^y_i|^{\beta})^2\\
 \leq \sum_{i=1}^n |y_i - a^y_i|^2\land |y_i - a^y_i|^{1+\beta} \leq 16\sum_{i=1}^n \Lb(y_i - a^y_i)/3\leq 16r/3
\end{multline*}
and similarly
\begin{multline*}
\|\nabla\Phi(y)\|_{(\beta+1)/\beta}^{1+\beta} \leq  \Big(\sum_{i=1}^n |y_i - a^y_i|^{(\beta+1)/\beta}\land |y_i - a^y_i|^{1+\beta}\Big)^{\beta}\\
\leq \Big(\sum_{i=1}^n |y_i - a^y_i|^{2}\land |y_i - a^y_i|^{1+\beta}\Big)^{\beta}
\leq (16\sum_{i=1}^n \Lb(y_i - a^y_i)/3)^{\beta} \leq 16r^{\beta}/3.
\end{multline*}
This finishes the estimation of the Lipschitz constants.

In order to show that for $r$ big enough we have the inclusion
\begin{equation*}
\{x\in\RR^n : \vphi(x)<r/16+\EE\vphi(X_1,\ldots,X_n) \} \subset \{x\in\RR^n : \vphi(x)<3r/16 \},
\end{equation*}
we claim that $ \EE \vphi(X_1,\ldots, X_n) \leq r/8$ for $r > 2^{12} C(\beta, m,\sigma)/3$. Indeed, if this was not the case, then we could write
\begin{align*}
1/2&\leq \PP(\vphi(X_1,\ldots, X_n) \leq 0)\leq \PP(\vphi(X_1,\ldots, X_n)
\leq \EE \vphi(X_1,\ldots, X_n) - r/8) \\
&\leq \PP(|\vphi(X_1,\ldots, X_n) - \EE \vphi(X_1,\ldots, X_n)|\geq r/8) \leq 64\Var \vphi(X_1,\ldots, X_n) /r^2 \\
&\leq  64 \cdot 16/3 \cdot 2C(\beta,m,\sigma)/r,
\end{align*}
where the first inequality follows from the fact that on the set $A$ we have $\vphi=\Phi=0$, and the last inequality follows from the fact that $\mu$ satisfies the Poincar\'e inequality (see Remark~\ref{rem:poinc}) and $\vphi$ is $\sqrt{16r/3}$-Lipschitz. This yields a contradiction.
\end{proof}

\section{A link to transport inequalities}\label{sec:transport}
In this section we will present several equivalent formulations of modified log-Sobolev inequalities for convex functions in terms of transportation inequalities or $\tau$-log-Sobolev inequalities, in the spirit of \cite[Theorem 6.8]{gozlan}. We would like to emphasize that our arguments in this part of the article follow quite closely the approach in \cite{gozlan} (in fact the authors of \cite{gozlan}, while working with the classical convex log-Sobolev inequality, i.e. $\beta = 1$, mention the possibility of extending their results to more general cost functions). In particular we do not claim any novelty of the methods we use.

Our main motivation for providing an equivalent formulation of modified log-Sobolev inequalities for all $\beta \in [0,1]$ is the hope that it may help find a characterization of measures satisfying the inequalities in question. In fact, one interesting observation which follows from our considerations is that modified log-Sobolev inequalities are equivalent to formally stronger results in which the entropy is replaced by a larger functional (see Remark~\ref{rem:stronger}).

Since we will rely heavily on techniques developed by other authors, we decided to present only the main steps and ideas of the proofs. We include numerous references to appropriate sources for the Reader interested in all technical details.

Before stating the actual results we have to introduce some notation which we will use throughout this section.

For $x=(x_1,\ldots, x_n)\in\RR^n$ and $\beta\in(0,1]$ (we will treat the case $\beta = 0$ later) we define
\begin{equation*}
\Hbn(x) = \sum_{i=1}^n \Hb(x_i), \quad \Lbn(x) = \sum_{i=1}^n \Lb(x_i),
\end{equation*}
where $\Lb$ is the Legendre transform of the function $\Hb(s) = s^2\lor |s|^{(\beta+1)/\beta}$ (see Definition~\ref{def:H-H*}).
For $t>0$, $\beta\in(0,1]$, and a function $f \colon \RR^n \to \RR$ we define the infimum-convolution operator
\begin{equation*}
Q_t f(x) = \inf_{y\in\RR^n} \{f(y) + t \Lbn((x-y)/t) \} \in [-\infty,\infty).
\end{equation*}

Let $P_1(\RR^n)$ denote the class of all probability measures on $\RR^n$ with finite first moment. If $\mu  \in P_1(\RR^n$), we define the transport cost
\begin{equation*}
\overline{\mathcal{T}}(\mu|\nu) = \inf_{\pi}\{ \int_{\RR^n} \Lbn\big(x - \int_{\RR^n} y dp_x(y)\big) d\nu(x) \},
\end{equation*}
where the infimum is taken over all probability measures $\pi$ on $\RR^n\times\RR^n$ with the first marginal equal to $\nu$ and the second to $\mu$; here $p_x$ denotes the probability kernel such that  $\pi(dx dy) = \nu(dx)p_x(dy)$.

If $f \colon \RR^n \to \RR$ is bounded below by an affine function, we define moreover
\begin{equation*}
R^{\lambda}f(x) = \inf_{p}\{ \int_{\RR^n} f(y)dp(y) + \lambda \Lbn\big(x - \int_{\RR^n} y dp(y)\big)\} \in [-\infty,\infty)
\end{equation*}
where the infimum is taken over all probability measures $p \in P_1(\RR^n)$. Note that $R^{\lambda}f\leq f$.

Recall also, that for two probability measures $\mu, \nu$ on $\RR^n$ the relative entropy of $\nu$ with respect to $\mu$ is given by the formula
\begin{equation*}
H(\nu|\mu) = \int_{\RR^n} \log\Big(\frac{d\nu}{d\mu}\Big) d\nu
\end{equation*}
if $\nu$ is absolutely continuous with respect to $\mu$; otherwise we set $H(\nu|\mu) = +\infty$.

The main result of this section provides a characterization of the modified log-Sobolev inequalities with $\beta \in (0,1]$ for convex functions in terms of transportation inequalities.
\begin{proposition}\label{prop:transport}
Let $X$ be an integrable random vector with values in $\RR^n$ and law $\mu$. For $\beta\in(0,1]$ the following conditions are equivalent.
\begin{enumerate}
\item[\itemi] There exists $b>0$ such that for all probability measures $\nu$ on $\RR^n$, $
\overline{\mathcal{T}}(\mu|\nu) \leq b H(\nu|\mu) $.
\item[\itemii]
For all $s > 0$ we have $\EE e^{s|X|} < \infty$
and
there exist $\lambda, D>0$ such that
\begin{equation}
\Ent e^{f(X) }\leq D \EE (f(X) - R^{\lambda}f(X))e^{f(X)}\label{eq:tau-lsi}
\end{equation}
for all functions $f\colon\RR^n\to\RR$ such that $\EE f(X)e^{f(X)}<\infty$ and $f$ is bounded below by an affine function.
\item[\itemiii]
For all $s > 0$ we have $\EE e^{s|X|} < \infty$
and
there exists $C>0$ such that
\begin{equation}
\Ent e^{\vphi(X) }\leq C \EE  \Hbn(\nabla \vphi(X))e^{\vphi(X)}\label{eq:mod-lsi-H}
\end{equation}
for all smooth convex Lipschitz functions $\vphi\colon\RR^n\to\RR$.

\item[\itemiv]
For all $s > 0$ we have $\EE e^{s|X|} < \infty$
and
there exists $B>0$ such that for every $t_0\in [0, B]$, every $t>0$, and every convex Lipschitz function $\vphi\colon\RR^n\to\RR$  bounded from below,
\begin{equation}\label{eq:hypercontr}
\big\| e^{Q_t \vphi(X)} \big\|_{k(t)} \leq  \big\| e^{\vphi(X)} \big\|_{k(0)},
\end{equation}
where $k(t) = (1+B^{-1}(t-t_0))\indbr{t\leq t_0} + (1+B^{-1}\beta^{-1} (t-t_0))^{\beta}\indbr{t>t_0}$.  Here, for a positive random variable $Z$, we use the notation $\|Z\|_0 = \exp( \EE \log Z )$ and $\|Z\|_k = (\EE Z^k)^{1/k}$ if $k\neq 0$ .
\end{enumerate}
Moreover, in each of the above implications the constants in the conclusion depend only on the constants in the premise (in particular they do not depend on the dimension).
\end{proposition}

\begin{remark}
Using the terminology and notation introduced in \cite{gozlan} the above proposition can be stated shorter: \itemi{} means that the transport inequality $\overline{\mathbf{T}}\vphantom{\mathbf{T}}_c^-(b)$ holds (for the cost function
$c(x,p) = \Lbn(x - \int_{\RR^n} y dp(y))$; see~\cite[Definition 5.1]{gozlan}) and \itemii{} is the $\tau$-log-Sobolev inequality, denoted $(\tau)-\mathbf{LSI}_c(\lambda, D)$ (for the same cost function; see\cite[Definition 6.1]{gozlan}).
\end{remark}

\begin{remark}
The exponential integrability of the norm in items \itemii{}-\itemiv{} is introduced to exclude heavy tailed measures for which the only exponentially integrable convex functions are constants. Using the observation that $R^\lambda f(x) = \inf_y \{\varphi(y) + \lambda \Lbn(x-y)$\}, where $\varphi$ is the greatest convex minorant of $f$ (see~\cite[Corollary 3.11 (2)]{gozlan}), one can easily see that such measures trivially satisfy \eqref{eq:tau-lsi} with $D = 1$ and arbitrary $\lambda > 0$, while they cannot satisfy the inequality of point~\itemi{}, as it implies subexponential concentration.
\end{remark}

\begin{remark}\label{rem:stronger}
As we will see from the proof (see~\eqref{eq:stronger}), the conditions of the proposition remain equivalent if we replace the entropy on the left-hand side of the inequalities in \itemii{} and \itemiii{} by the (greater) expression
\begin{equation*}
\int_{\RR^n} \vphi e^{\vphi}d\mu - \int_{\RR^n} \vphi d\mu \int_{\RR^n} e^{\vphi}d\mu
 =
\frac{1}{2} \EE (\vphi(X) - \vphi(Y))(e^{\vphi(X)} - e^{\vphi(Y)})
\end{equation*}
(where $X,Y$ are independent with law $\mu$). Therefore the first estimate from the proof of Theorem~\ref{thm:main} is actually quite sharp. Note that there does not exist a finite constant $C$ such that $C \Ent e^X \ge \EE (X-Y)(e^X-e^Y)$ for all bounded i.i.d. random variables $X$ and $Y$ (as one can see by considering e.g. $X$ with the distribution $\frac{1}{2}\delta_{-a} + \frac{1}{2}\delta_{a}$ for $a$ sufficiently large).
\end{remark}

\begin{remark}\label{rem:dual} In the proofs we will use a dual formulation of   \itemi{}: the transport cost inequality holds if and only if  we have
\begin{equation}\label{eq:dual-formulation}
\big(\int_{\RR^n} e^{b^{-1}Q_1\vphi(x)} d\mu(x) \big)^{b}  e^{-\int_{\RR^n} \vphi(x) d\mu(x)}\leq 1
\end{equation}
 for all convex Lipschitz functions $\vphi\colon\RR^n\to\RR$ bounded from below
(see \cite[Proposition~5.5~(ii'') and Corollary~3.11~(2)]{gozlan}; note also that $Q_1 \vphi = R^1\vphi$ since $\vphi$ is convex).
\end{remark}

For $\beta=0$ a result similar to Proposition \ref{prop:transport} holds, but we have to introduce a slightly different notation which will additionally depend on a number $\delta>0$. Namely, for $x=(x_1,\ldots, x_n)\in\RR^n$ we define
\begin{equation*}
\Hzerodeltan{\delta}(x) = \sum_{i=1}^n \Hzerodelta{\delta}(x_i), \quad \Lzerodeltan{\delta}(x) = \sum_{i=1}^n \Lzerodelta{\delta}(x_i),
\end{equation*}
where $\Hzerodelta{\delta}(s) = \Hzero(s/\delta) = \delta^{-2} s^2\indbr{|s|\leq \delta} + \infty\indbr{|s|>\delta}$ and $\Lzerodelta{\delta}(s) =\Lzero(s\delta) =   \frac{1}{4}\delta^{2} s^2\indbr{\delta|s|\leq 2 } + (\delta|s|-1)\indbr{\delta|s|>2}$ is its Legendre transform. We define $Q_{t,\delta}$, $\overline{\mathcal{T}}_{\delta}$, $R^{\lambda}_{\delta}$ as $Q_t$, $\overline{\mathcal{T}}$, $R^{\lambda}$ above, but with $\Hzerodeltan{\delta}$ and $\Lzerodeltan{\delta}$ in place of $\Hbn$ and $\Lbn$ respectively, i.e. we put
\begin{align*}
Q_{t,\delta} f(x)
	&= \inf_{y\in\RR^n} \{f(y) + t \Lzerodeltan{\delta}((x-y)/t) \},\\
\overline{\mathcal{T}}_{\delta}(\mu|\nu)
	&= \inf_{\pi}\{ \int_{\RR^n} \Lzerodeltan{\delta}\big(x - \int_{\RR^n} y dp_x(y)\big) d\nu(x) \},\\
R^{\lambda}_{\delta}f(x)
	&= \inf_{p \in P_1(\RR^n)}\{ \int_{\RR^n} f(y)dp(y) + \lambda \Lzerodeltan{\delta}\big(x - \int_{\RR^n} y dp(y)\big)\}.
\end{align*}

We then have the following
\begin{proposition}\label{prop:transport-zero}
Let $X$ be an integrable random vector with values in $\RR^n$ and law $\mu$. The following conditions are equivalent.
\begin{enumerate}
\item[\itemi] There exist $\delta_1$ and $b>0$ such that for all probability measures $\nu$ on $\RR^n$,
$\overline{\mathcal{T}}_{\delta_1}(\mu|\nu) \leq b H(\nu|\mu) $.
\item[\itemii] There exist $\delta_2>0, \lambda, D>0$ such that
$
\EE e^{D^{-1} \Lzerodeltan{\delta_2}(X)} < \infty
$
and
\begin{equation*}
\Ent e^{f(X) }\leq D \EE (f(X) - R^{\lambda}_{\delta_2}f(X))e^{f(X)}
\end{equation*}
for all functions $f\colon\RR^n\to\RR$ such that $\EE f(X)e^{f(X)}<\infty$ and $f$ is bounded below by an affine function.
\item[\itemiii]
There exist $\delta_3, C>0$ such that
$
\EE e^{C^{-1} \Lzerodeltan{\delta_3}(X)} < \infty
$
and
\begin{equation*}
\Ent e^{\vphi(X) }\leq C \EE  \Hzerodeltan{\delta_3}(\nabla \vphi(X))e^{\vphi(X)}
\end{equation*}
for all smooth convex Lipschitz functions $\vphi\colon\RR^n\to\RR$ with $\|\nabla \vphi \|_{\infty}\leq \delta_3$.

\item[\itemiv]
There exist $\delta_4, B>0$ such that
$
\EE e^{B^{-1} \Lzerodeltan{\delta_4}(X)} < \infty
$
and for every $t_0\in [0, B]$, every $t>0$, and every convex Lipschitz function $\vphi\colon\RR^n\to\RR$  bounded from below,
\begin{equation*}
\big\| e^{Q_{t,\delta_4} \vphi(X)} \big\|_{k(t)} \leq  \big\| e^{\vphi(X)} \big\|_{k(0)},
\end{equation*}
where $k(t) = B^{-1}\min\{1+B^{-1}(t-t_0) , 1\}$.

\end{enumerate}
Moreover, in each of the above implications the constants in the conclusion depend only on the constants in the premise (in particular they do not depend on the dimension).
\end{proposition}

\begin{remark} The dependence of the constants is the following: \itemi{} implies \itemii{} with  $\lambda\in(0,1/b)$, $D=1/(1 - \lambda b)$, and $\delta_2=\delta_1$; \itemii{} implies \itemiii{} with $C=\lambda D$ and $\delta_3=\delta_2 (\lambda \wedge (2D)^{-1})$; \itemiii{} implies \itemiv{} with $B=C\vee 1$ and $\delta_4=\delta_3$; \itemiv{} implies \itemi{} with $b=B$ and $\delta_1=\delta_4$. Note also that for $a\geq 1$ we have $\Hzerodelta{a\delta}\leq \Hzerodelta{\delta}$ and $\Lzerodelta{\delta} \leq H^*_{0,a\delta} \leq a^2 \Lzerodelta{\delta}$, and hence if the inequalities in \itemi{}, \itemii{}, \itemiii{}, \itemiv{} hold for some $\delta>0$, then they also hold for every $\delta' \in (0,\delta)$.
\end{remark}

\begin{proof}[Sketch of the proof of Proposition~\ref{prop:transport}]
For simplicity we assume $n=1$. To prove that \itemi{} implies \itemii{} we follow the proof of \cite[Proposition~6.3]{gozlan}. Fix $\lambda\in(0,1/b)$ and let $f\colon\RR\to\RR$ be a function with $\int_{\RR} fe^f d\mu<\infty$. We define $\nu_f$ to be the measure with density $e^f d\mu/(\int_{\RR} e^f d\mu)$. Let $\pi$ be a measure on $\RR^2$ with the first marginal $\nu_f$ and the second $\mu$. We have $\pi(dxdy) = \nu_f(dx)p_x(dy)$ for some probability kernel $p_x$ and hence
\begin{align*}
\int_{\RR} f d\nu_f - \int_{\RR} f d\mu& =	 \int_{\RR^2} (f(x)-f(y)) \pi(dxdy) =
 \int_{\RR} \big( f(x) - \int_{\RR} f(y) dp_x(y)\big) d\nu_f(x)\\
 &\leq
 \int_{\RR} (f(x) - R^{\lambda} f(x) ) d\nu_f(x) + \lambda  \int_{\RR} \Lb\big(x - \int_\RR y dp_x(y)\big)d\nu_f(x),
\end{align*}
where the last inequality follows from the definition of $R^{\lambda}$.
Optimizing over all measures $\pi$ as above gives us
\begin{align*}
\int_{\RR} f d\nu_f - \int_{\RR} f d\mu &\leq
 \int_{\RR} (f(x) - R^{\lambda} f(x) ) d\nu_f(x) + \lambda  \overline{\mathcal{T}}(\mu|\nu_f)\\
&\leq \int_{\RR} (f(x) - R^{\lambda} f(x) ) d\nu_f(x) + \lambda b  H(\nu_f|\mu).
\end{align*}
It follows from Jensen's inequality that
\begin{equation*}
H(\nu_f|\mu) = \int_{\RR} \log\Big( \frac{e^f}{\int_{\RR} e^f d\mu}\Big) d\nu_f \leq
\int_{\RR} f d\nu_f - \int_{\RR} f d\mu
\end{equation*}
and hence we conclude from the preceding inequality that
\begin{align*}
\int_{\RR} f d\nu_f - \int_{\RR} f d\mu \leq
 \frac{1}{1-\lambda b}\int_{\RR} (f(x) - R^{\lambda} f(x) ) d\nu_f(x) .
\end{align*}
By the definition of $\nu_f$ this means that
\begin{align}
\int_{\RR} f e^fd\mu - \int_{\RR} f d\mu \int_{\RR} e^f d\mu \leq
 \frac{1}{1-\lambda b}\int_{\RR} (f(x) - R^{\lambda} f(x) ) e^fd\mu(x).\label{eq:stronger}
\end{align}
Therefore \eqref{eq:tau-lsi} is satisfied with $D = 1/(1-\lambda b)$ (the above inequality is formally stronger, since its left-hand side is by Jensen's inequality greater than $\Ent e^{f}$; see Remark~\ref{rem:stronger}).

It remains to prove the exponential integrability of the norm (we present the argument on the real line, however it can be easily adapted to $\RR^n$, e.g. if we decide to work with the $\ell_1^n$ norm all the calculations are performed on each coordinate separately; clearly the choice of the norm for this problem is irrelevant).

Consider the function $\varphi(x) = r|x|$. By the dual formulation of the transport cost inequality (see Remark~\ref{rem:dual}) and the assumption that $X$ is integrable, we get $\EE e^{b^{-1}Q_1 \varphi(X)} < \infty$.
It is easy to see that for $|x|$ large enough (depending on $r$) the infimum in the definition of $Q_1\varphi(x)$ is attained for $y = x - \sgn(x) ((\Lb) ')^{-1}(r)$, where $((\Lb)')^{-1}$ is the inverse of $(\Lb)'$ restricted to $[0,\infty)$.
Thus for $|x|$ large enough, we have $Q_1 \varphi(x) \ge r|x|/2$. Since $r$ is arbitrary, we conclude that $\EE e^{s|X|} < \infty$ for all $s$.

To prove that \itemii{} implies \itemiii{} note that if $\vphi$ is convex, then
\begin{align*}
\vphi(x) - R^{\lambda}\vphi(x) &= \sup_p \{ \int_{\RR} \vphi(x)-\vphi(y)dp(y) - \lambda \Lb\big(x - \int_\RR y dp(y)\big)\}\\
&\leq \sup_p \{ \int_{\RR} \vphi'(x)(x-y)dp(y) - \lambda \Lb\big(x - \int_\RR y dp(y)\big)\}\\
&\leq \sup_p \{ \vphi'(x) \cdot\big(x - \int_\RR y dp(y)\big) - \lambda \Lb\big(x - \int_\RR y dp(y)\big)\}.
\end{align*}
We have $uv\leq C \Hb(u) + \lambda \Lb(v)$
for some constant $C$ depending only on $\lambda$ and $\beta$.
Therefore
\begin{align*}
\vphi(x) - R^{\lambda}\vphi(x) \leq C \Hb(\vphi'(x))
\end{align*}
and hence \itemiii{} holds.

To see that \itemiii{} implies \itemiv{} (with $B=C$) we follow the reasoning from the proof of~\cite[Theorem 1.11]{MR3186934}. By a perturbation argument we can assume that $\mu$ is absolutely continuous. Indeed, if $\gamma$ is a Gaussian measure on $\RR^n$ with the covariance matrix being a sufficiently small multiple of identity, then the product measure $\mu\otimes \gamma$ on $\RR^n \times \RR^n$ satisfies the modified log-Sobolev inequality with constant $C$.  Let $\vphi:\RR^n \to \RR^n$ be a smooth convex Lipschitz function and let $\Phi:\RR^n\times\RR^n\to \RR^n$ be defined by the formula $\Phi (x,y) = \varphi(x+\eps y)$, $x,y\in\RR$. Applying the modified log-Sobolev inequality to the measure $\mu\otimes \gamma$ and the function $\Phi$, we see  that the convolution $\mu*\gamma_\varepsilon$, where $\gamma_\varepsilon(A) = \gamma(A/\varepsilon)$, satisfies the modified log-Sobolev inequality on $\RR^n$ with constant $C_\varepsilon \to C$ as $\varepsilon \to 0$. If one can prove that \eqref{eq:hypercontr} is satisfied by $\mu*\gamma_\varepsilon$ with $B_\varepsilon  = C_\varepsilon$, then one can obtain it for $\mu$ with $B= C$ by the Lebesgue dominated convergence theorem. Note that if $\mu$ is absolutely continuous, then another standard approximation shows that \itemiii{} holds for all convex Lipschitz functions (by the Rademacher theorem the gradient is then almost surely well defined).

If $\vphi\colon\RR\to\RR$ is a convex function, then so is $Q_t\vphi$ and it satisfies the Hamilton-Jacobi equation
\begin{equation}
{\partial_t} Q_t\vphi + \Hb(\partial_x Q_t\vphi ) = 0 \label{eq:ham-jac}
\end{equation}
(see~\cite[Chapter~3.3.2]{evans}). More precisely, the function $(t,x) \mapsto Q_t\vphi(x)$ is Lipschitz and the equation is satisfied a.e. with respect to the Lebesgue measure on $(0,\infty)\times \RR^n$.
We fix $t_0\in[0,C]$, define
\begin{equation*}
k(t) = (1+C^{-1}(t-t_0))\indbr{t\leq t_0} + (1+C^{-1}\beta^{-1} (t-t_0))^{\beta}\indbr{t>t_0},
\end{equation*}
and  set $F(t) = \int_{\RR} e^{k(t)Q_t\vphi(x)}d\mu(x)$
for $t>0$. Using the absolute continuity of $\mu$ together with integrability properties of Lipschitz functions, one can show that $F$ is locally Lipschitz, $F'(t)$ exists for almost all $t>0$, and

\begin{align*}
k(t)F'(t) &=
k(t) \int_{\RR} e^{k(t)Q_t\vphi(x)} \big( k'(t) Q_t\vphi(x) + k(t)\partial_t Q_t\vphi(x)\big) d\mu(x)\\
&= k(t)\int_{\RR} e^{k(t)Q_t\vphi(x)} \big( k'(t) Q_t\vphi(x) - k(t) \Hb(\partial_x Q_t\vphi(x))\big) d\mu(x)\\
&=
k'(t) F(t)\log F(t) + k'(t)\Ent e^{k(t)Q_t\vphi(X)} \\
\MoveEqLeft[-12] - k^2(t) \int_{\RR} e^{k(t)Q_t\vphi(x)}\Hb(\partial_x Q_t\vphi(x)) d\mu(x)\\
&\leq
k'(t) F(t)\log F(t) + C k'(t)\int_{\RR} e^{k(t)Q_t\vphi(x)}\Hb(k(t)\partial_x Q_t\vphi(x)) d\mu(x)\\
\MoveEqLeft[-12]
- k^2(t) \int_{\RR} e^{k(t)Q_t\vphi(x)}\Hb(\partial_x Q_t\vphi(x)) d\mu(x)\\
&\leq
 k'(t) F(t)\log F(t)\\
\MoveEqLeft[-1]
 + \int_{\RR} e^{k(t)Q_t\vphi(x)}\Hb(\partial_x Q_t\vphi(x)) \cdot\big(Ck'(t)\max\{ |k(t)|^2, |k(t)|^{\frac{\beta+1}{\beta}}\} - k^2(t) \big)\\
& =  k'(t) F(t)\log F(t).
\end{align*}

where we used \eqref{eq:ham-jac}, the inequality~\eqref{eq:mod-lsi-H}, the fact that
\begin{equation*}
\Hb(a x) \leq \max\{|a|^2, |a|^{(\beta+1)/\beta}\}\Hb(x),
\end{equation*}
and the definition of $k(t)$ in the last equality.
 The above differential inequality is equivalent to $(\log(F(t))/k(t))'\leq 0$ for almost all $t>0$. Since $Q_t\vphi\leq\vphi$, we arrive at

\begin{equation*}
\frac{\log F(t)}{k(t)} \leq  \liminf_{s\to 0^+} \frac{\log F(s)}{k(s)} \leq \lim_{s\to 0^+} \frac{\log( \int_{\RR} e^{k(s)\vphi(x)}d\mu(x))}{k(s)} = \log \| e^{\vphi(X)}\|_{k(0)},
\end{equation*}
which is equivalent to the assertion.

Finally, \itemiv{} implies \itemi{} (with $b=B\vee 1$). Indeed, we plug $t_0 = B$ and $t=B\land 1\le 1$ (so that $k(t) = (B\vee 1)^{-1}$ and $k(0)=0$) into inequality~\eqref{eq:hypercontr}, and using the fact that $Q_t\varphi$ decreases with $t$ we arrive at a dual formulation of the transport cost inequality (see Remark~\ref{rem:dual}).  This ends the proof of the last implication.
\end{proof}

The proof in the case $\beta=0$ follows roughly the sames lines. Below we sum up the more important changes one has to make.

\begin{proof}[Sketch of the proof of Proposition~\ref{prop:transport-zero}]
For simplicity we assume $n=1$. To prove that \itemi{} implies the entropy bound of \itemii{} (with $\delta_2=\delta_1$) we follow exactly the same reasoning as above, in the proof of Proposition~\ref{prop:transport}. As for the exponential integrability of $\Lzerodeltan{\delta_1}$, we again consider $\varphi(x) = r|x|$ and note that by the dual formulation of \itemi{}, $\EE e^{b^{-1} Q_{1,\delta_1} \varphi(X)} < \infty$. But, as one can easily see, the boundedness of the derivative of $\Lzerodelta{\delta_1}$ implies that for $r$ sufficiently large $Q_{1,\delta_1} \varphi(x) = \inf_y\{ r|y| + \Lzerodelta{\delta_1}(x-y)\} = \Lzerodelta{\delta_1}(x)$, which shows that $\EE e^{b^{-1} \Lzerodelta{\delta_1}(X)} < \infty$.

To prove that \itemii{} implies \itemiii{} with $\delta_3=\delta_2(\lambda \wedge (2D)^{-1})$ and $C=\lambda D$, notice that if $\vphi$ is convex and $|\vphi'|\leq \delta_3$, then
$\EE \varphi(X)e^{\varphi(X)} < \infty$ and
\begin{equation*}
\vphi(x) - R^{\lambda}_{\delta_2}\vphi(x) \leq \sup_p \{ \vphi'(x) \cdot\big(x - \int_\RR y dp(y)\big) - \lambda \Lzerodelta{\delta_2}\big(x - \int_\RR y dp(y)\big)\}.
\end{equation*}
We have
\begin{align*}
uv& \leq (\lambda\Lzerodelta{\delta_2})^*(u) + \lambda \Lzerodelta{\delta_2}(v) =\lambda \Hzerodelta{\delta_2}(u/\lambda) + \lambda \Lzerodelta{\delta_2}(v) \\
&= \lambda\Hzerodelta{\lambda\delta_2}(u) + \lambda \Lzerodelta{\delta_2}(v) \le \lambda \Hzerodelta{\delta_3}(u) + \lambda \Lzerodelta{\delta_2}(v).
\end{align*}
Therefore $\vphi(x) - R^{\lambda}_{\delta_2}\vphi(x) \leq \lambda  \Hzerodelta{\delta_3}(\vphi'(x))$ and hence \itemiii{} holds.

As in the case $\beta > 0$, for the proof that \itemiii{} implies \itemiv{} (with $\delta_4 = \delta_3$ and $B=C\vee 1$) we can assume that $\mu$ is absolutely continuous. Thus, by standard approximation arguments, the log-Sobolev inequality of \itemiii{} is satisfied for all convex Lipschitz functions (the gradient is then almost surely well defined). In what follows, without loss of generality, we will also assume that $C \ge 1$.

Note that $|\partial_x t\Lzerodelta{\delta_3}((x-y)/t)| = |(\Lzerodelta{\delta_3})'((x-y)/t)|\leq \delta_3$, so for $t > 0$, $x\mapsto Q_{t,\delta_3}\varphi(x)$ is $\delta_3$-Lipschitz. An adaptation of the reasoning from~\cite[Chapter~3.3.2]{evans} shows that $(t,x)\mapsto Q_{t,\delta_3}\varphi(x)$ is Lipschitz and
\begin{equation*}
{\partial_t} Q_{t,\delta_3}\vphi + H_{0,\delta_3}(\partial_x Q_{t,\delta_3}\vphi) = {\partial_t} Q_{t,\delta_3}\vphi + \delta_3^{-2}(\partial_x Q_{t,\delta_3}\vphi )^2 = 0
\end{equation*}
Lebesgue a.e. on $(0,\infty)\times \RR$ (we remark that for $n > 1$ one needs to work  with the $\ell_1^n$ rather than with the Euclidean norm to get that $t \mapsto Q_{t,\delta_3} \varphi(x)$ is Lipschitz).

Next, one shows that for any $\tilde{C}\ge C$,
 the function  $F(t) = \int_{\RR} e^{\tilde{k}(t)Q_{t,\delta_3}\vphi(x)}d\mu(x)$, $t>0$, is well defined,  where we have denoted
\begin{equation*}
\tilde{k}(t) = \tilde{C}^{-1}\min\{ (1+\tilde{C}^{-1}(t-t_0)), 1\},
\end{equation*}
and $t_0\in[0,\tilde{C}]$ is fixed number. Indeed,
\begin{equation*}
\tilde{k}(t) Q_{t,\delta_3}\vphi(x) \leq \tilde{k}(t)  t \Lzerodelta{\delta_3} (x/t)+\tilde{k}(t)\vphi(0)\leq \tilde{k}(t) \delta_3|x|+\tilde{k}(t)\vphi(0).
\end{equation*}
For $|x|>2/\delta_3$ the right-hand side above is bounded by $\tilde{C}^{-1} (\Lzerodelta{\delta_3}(x)+1+\vphi(0))$,
which together with the assumed exponential integrability of $\Lzerodelta{\delta_3}$ shows that the integral defining $F$ is finite.

Using absolute continuity of $F$ and again exponential integrability of $\Lzerodelta{\delta_3}$ one proves that for $\tilde{C} > C$, $F$ is differentiable a.e. Moreover, the same argument as for $\beta > 0$ shows  that $F$ satisfies $\tilde{k}(t)F'(t)\leq \tilde{k}'(t)F(t)\log  F(t)$ for almost all $t>0$. Let us list the additional observations needed to carry out the calculations. First, $|\partial_x \tilde{k}(t) Q_{t,\delta_3}\vphi| \leq|\partial_x Q_{t,\delta_3}\vphi|  \leq \delta_3$, which allows us to use the modified log-Sobolev inequality. Second, we have
$H_{0,\delta_3}(\tilde{k}(t)\partial_x Q_{t,\delta_3}\varphi(x)) = \tilde{k}(t)^2 H_{0,\delta_3}(\partial_x Q_{t,\delta_3}\varphi(x))$. Finally $\tilde{C}\tilde{k}'(t) \tilde{k}(t)^2 \le \tilde{k}(t)^2$. We remark that to verify the above properties one uses the assumption that $C \ge 1$.

Integration of  the differential inequality and taking the limit $\tilde{C} \to C$ ends the proof of the implication.

Finally, \itemiv{} implies \itemi{} (with $b=B$ and $\delta_1=\delta_4$) as in the case $\beta>0$: it is enough to check that
\begin{equation*}
\big(\int_{\RR} e^{B^{-1}Q_{1,\delta_3}\vphi(x)} d\mu(x) \big)^{B}  e^{-\int_{\RR} \vphi(x) d\mu(x)}\leq 1
\end{equation*}
holds for all convex Lipschitz functions $\vphi\colon\RR\to\RR$ bounded from below (see Remark~\ref{rem:dual}). This is easy to recover from the hypercontractive inequality in \itemiv{}. If $B < 1$ one sets $t_0 = t = B$ and uses the inequality $Q_{1,\delta_3}\vphi \le Q_{t,\delta_3}\vphi$, otherwise one sets $t = t_0 = 1$.\end{proof}

\begin{remark}\label{rem:Feldheim}
As mentioned in the Introduction, very recently Feldheim et al. \cite{Feldheim} proved that in the case of symmetric measures on $\RR$, the condition $\mu \in \mzero$ for some $m,\sigma$ is equivalent to the infimum convolution inequality
\begin{displaymath}
\int_\RR e^{Q_{1,\delta} \varphi (x)} d\mu(x) \int_\RR e^{- \varphi(x)}d\mu(x) \le 1
\end{displaymath}
for some $\delta$, which is a formally stronger property than the inequality \itemi{} of Proposition \ref{prop:transport-zero} (see Remark~\ref{rem:dual}). Thus for $\beta = 0$ the modified log-Sobolev inequality is in fact equivalent to a stronger concentration property.
Gozlan et al. in \cite{gozlan_new} extended the characterization of \cite{Feldheim} to much more general transport costs, including all the cases considered by us. It would be interesting to provide an intrinsic characterization also for measures on the line satisfying general modified log-Sobolev inequalities. One expects that for $\beta \in (0,1]$ such measures form a strictly larger class than the class of measures satisfying the inequality of \cite{gozlan_new}.
\end{remark}

\section{A remark about the non-convex setting}\label{sec:remark-non-convex}

As we noted in Remark~\ref{rem:poinc}, a Borel probability measures satisfying inequality \eqref{eq:mod-lsi-zero} for all smooth convex functions (with bounded derivatives), satisfies also the Poincar\'e inequality for all smooth convex functions. It follows from \cite[Theorem~1.4]{bobkov-goetze} that in this case $\mu\in\mzero$ for some $m$ and $\sigma$.

We will now show that sufficiently nice Borel probability measures $\mu$ which satisfy the modified logarithmic Sobolev inequality \eqref{eq:mod-lsi-beta} for all (not necessarily log-convex)  smooth functions belong to the class $\mbeta$ for some $m$ and $\sigma$. More precisely, for $\beta>0$ we have the following result.

\begin{proposition}
Le $\mu$ be a Borel probability measure on $\RR$, absolutely continuous with respect to the Lebesgue measure. Suppose there exist constants $C<\infty$ and $\beta\in(0,1]$ such that for all smooth functions
\begin{equation*}
\Ent f(X)^2 \leq C \EE (f'(X)^2 \lor |f'(X)|^{\frac{\beta+1}{\beta}}),
\end{equation*}
where $X$ is a random variable with law $\mu$. Denote by $M$ the median  and by $n$ the density of $\mu$, and assume also that
\begin{equation*}
\frac{1}{\beta} n(x)^{-\beta}\geq \eps \int_{x\land M}^{x\lor M} n(t)^{-\beta}dt
 \end{equation*}
 for some $\eps>0$ and all $x\neq M$.  Then there exist constants $m, 	\sigma<\infty$, such that $\mu\in\mbeta$.
\end{proposition}

\begin{proof}
It follows from the Barthe-Roberto criterion (see~\cite[Theorem 10]{barthe-roberto})  that for some constant $K$ and all $y\geq M$,
\begin{equation}\label{eq:last}
\mu([y,\infty))\log \frac{1}{\mu([y,\infty))} \Big(\int_M^y n(t)^{-\beta}dt\Big)^{1/\beta} \leq K.
\end{equation}
H\"older's inequality (for exponents $\beta+1$ and $(\beta+1)/\beta$) gives us
\begin{multline*}
\frac{1}{x^{\beta}} =  \int_x^{x+1/x^{\beta}} \negativespace  n(t)^{-\beta/(\beta+1)} n(t)^{\beta/(\beta+1)} dt\\
 \leq \Big(\int_x^{x+1/x^{\beta}} \negativespace n(t)^{-\beta} dt\Big)^{1/(\beta+1)}\Big(\int_x^{x+1/x^{\beta}} \negativespace n(t)dt \Big)^{\beta/(\beta+1)},
\end{multline*}
and therefore
\begin{equation*}
\frac{1}{x^{1+\beta}}  \leq
\Big(\int_x^{x+1/x^{\beta}} \negativespace n(t)^{-\beta} dt\Big)^{1/\beta}\mu([x, x+1/x^{\beta})).
\end{equation*}
Since $\mu$ has the concentration property of order $1+\beta$, using \eqref{eq:last} for $y=x+1/x^{\beta}$ we get that
\begin{align*}
\frac{D}{K} x^{1+\beta} \mu([x+1/x^{\beta},\infty)) \leq \Big(\int_M^{x+1/x^{\beta}} \negativespace n(t)^{-\beta}dt\Big)^{-1/\beta}\leq x^{1+\beta}\mu([x, x+1/x^{\beta}))
\end{align*}
 for some constant $D$ and $x$ large enough. Hence, for $x$ large enough,
\begin{equation*}
\mu([x+1/x^{\beta},\infty)) \leq \frac{1}{1+D/K} \mu([x+1/x^{\beta},\infty)).
\end{equation*}
Since we can deal analogously with the left tail, the claim follows from Proposition~\ref{prop:equiv-class-mbeta}.
\end{proof}

\begin{remark}As proved in \cite{MR2163396}, in the case of $\beta = 1$ we do not need the assumption of absolute continuity, as in this case the Barthe-Roberto criterion can be replaced by a more general Bobkov-G\"otze criterion \cite{MR1682772}, in which one does not require the existence of a density.
\end{remark}

\bibliographystyle{amsplain}	
\providecommand{\bysame}{\leavevmode\hbox to3em{\hrulefill}\thinspace}
\providecommand{\MR}{\relax\ifhmode\unskip\space\fi MR }
\providecommand{\MRhref}[2]{%
  \href{http://www.ams.org/mathscinet-getitem?mr=#1}{#2}
}
\providecommand{\href}[2]{#2}

\end{document}